\documentclass[11pt,reqno]{amsart} 

\usepackage{amssymb,latexsym}
\usepackage{cite} 

\usepackage{fontenc}
\usepackage{times}

\usepackage{amsfonts,amssymb,epsfig}
\usepackage{amscd}
\usepackage{color}
 \usepackage{hyperref}

\usepackage{textcomp}
\usepackage{mathrsfs}

\usepackage[height=200mm,width=140mm]{geometry} 

\theoremstyle{plain}
\newtheorem{theorem}{Theorem}
\newtheorem{lemma}{Lemma}
\newtheorem{cor}{Corollary}
\newtheorem{pro}{Proposition}

\theoremstyle{definition}
\newtheorem{definition}{Definition}

\newenvironment{prof1a}[1][Proof of Proposition \ref{pro1}]{\textbf{#1.} }{\ \rule{0.5em}{0.5em}}

\theoremstyle{remark}
\newtheorem{rem}{Remark}
\theoremstyle{example}
\newtheorem{exe}{Examples}

\numberwithin{equation}{section} 

\begin{document}
\title[Finslerian metrics locally conformally $R$-Einstein]{Finslerian metrics locally conformally $R$-Einstein} 
\author{Serge Degla}

\address{Ecole Normale Sup\'erieure de Natitingou\\ 
\noindent
P.~O. Box 72 \\ 
 Natitingou\\ B\'enin}

 \email{deglaserge@yahoo.fr}

\author{Gilbert Nibaruta}
\address{Ecole Normale Sup\'erieure\\ Section de Math\'ematiques \\  P.~O. Box 6983\\
  Bujumbura-Burundi}
\email{gilbert.nibaruta@imsp-uac.org}

\author{L\'eonard Todjihounde}
\address{Universit\'e d'Abomey-Calavi\\ Institut de Math\'ematiques et de Sciences Physiques \\  P.~O. Box 613\\
  Porto-Novo\\ B\'enin}

\email{leonardt@imsp-uac.org}

\begin{abstract}
Let $R$ be the $hh$-curvature associated with the Chern connection or the Cartan connection. 
Adopting the pulled-back tangent bundle approach to the Finslerian Geometry, 
an intrinsic characterization of $R$-Einstein metrics is given. 
Finslerian metrics which are locally conformally $R$-Einstein are classified.
\end{abstract} 


\maketitle

\section{Introduction}\label{Section1}
Finslerian metrics are of considerable interest due to their
rich structure including Riemann, Randers, Landsberg and  
Berwald type metrics. Some areas in which they have significant impacts are Differential Geometry, 
Einstein's theory of General Relativity and Biology \cite{biAntonelli,biBao}. 
A natural and important problem is the classification of metrics conformally Einstein.
 In 1923, Brinkmann obtained in \cite{biBrinkmann1} the necessary and sufficient conditions for an $n$-dimensional 
Riemannian manifold to be conformally Einstein. Later, 
Szekeres \cite{biSzekeres} in 1963, Kozameh-Newmann-Tod \cite{biKNT} in 1985, 
Listing \cite{biListing1} in 2001, Gover-Nurowski \cite{biGover} in 2005, as well as K\"{u}hnel-Rademacher \cite{biRademacher} in 2016 
studied this problem from different points of view, both 
for (pseudo-)Riemannian metrics. This motivates us to study the above problem for a general Finslerian metric.

In the present paper, we study and characterize Finslerian metrics which are locally conformal to $R$-Einstein metrics. 
Unfortunately, the specificity of the Finslerian metric
 and his associated fundamental tensor do not allow us to use
the same technics and tools as in the Riemannian case to
obtain general classifications of (locally or globally) conformally Finslerian $R$-Einstein metrics. Hence, 
we exploit the pulled-back bundle approach 
and introduce a globally theory on 
conformal Finslerian $R$-Einstein geometry. Let $M$ be an $n$-dimensional $C^{\infty}$ connected manifold 
and $\mathring{T}M:=TM \backslash\{0\}$ its slit tangent bundle. The submersion $\pi:\mathring{T}M\longrightarrow M$ 
pulls back the tangent bundle $TM$ to a vector bundle $\pi^*TM$ over $\mathring{T}M$. 
Given a Finslerian metric $F$ on $M$ and $g$ its fundamental tensor, we have introduced in \cite{biNibaruta1}, the 
following tensor. The trace-free horizontal Ricci tensor of a Finslerian manifold $(M,F)$
is the application
$$  \textbf{E}_F^H:\begin{matrix}
\Gamma(\pi^*TM)\times  \chi(\mathring{T}M)&\to&C^{\infty}(\mathring{T}M,\mathbb{R})\\
 (\xi,X)&\mapsto
 &(\textbf{Ric}_F^H-\frac{1}{n}\textbf{Scal}_F^H\underline{g})(\xi,X) 
 \end{matrix}$$
 where $\textbf{Ric}_F^H$ is the horizontal Ricci tensor, $\textbf{Scal}_F^H$ 
 is the horizontal scalar curvature and $\underline{g}:=\pi^*g$ is the pullback of $g$ by the submersion 
  $\pi:\mathring{T}M \longrightarrow M$. 
 One of advantage of the tensor 
$\textbf{E}_F^H$, it vanishes when $F$ is an $R$-Einstein metric. Furthermore, it is insensitive to whether we use 
the Chern connection or the Cartan connection. 
Our main results in this work are given by the following.
\begin{pro}\label{pro1}
Let $F$ be a Finslerian metric on an $n$-dimensional manifold. 
 $F$ is locally conformal to an $R$-Einstein metric $\widetilde{F}$,
 with $\widetilde{F}=e^uF$, 
 if and only if the conformal factor $e^u$ is a solution of the equation
 \begin{eqnarray}
 &&\textbf{E}_F(\partial_i,\hat{\partial}_j)-(n-2)\left(\nabla_j\nabla_iu-\nabla_iu\nabla_ju\right)\nonumber\\
 &&+\frac{(n-2)}{n}\left(\nabla^d\nabla_du-\nabla^du\nabla_du\right)g_{ij}\nonumber\\
                  &&+\frac{(n-1)}{2nF}\left(\nabla_ru\nabla^{q}u\right)\frac{\partial(F^2g^{rs}-2y^ry^s)}{\partial y^q}
                  g^{kl}\mathcal{A}_{skl}g_{ij}\nonumber\\
                  &=&0.\label{108cd}
        \end{eqnarray}
\end{pro}

To determine the solution(s) of the equation (\ref{108cd}), we consider it as a system of partial differential equations in 
 the conformal factor $e^u$ 
and curvatures associated with $F$ on a neighborhood of the given manifold. 
The explicit solution $u$ can tell us how $F$ is constructed.
Hence, we prove the following.
\begin{theorem}\label{theo1a1}
         A Finslerian metric $F$ on a $2$-dimensional manifold is locally conformally $R$-Einstein
         if and only if one of the following two cases holds:
         \begin{itemize}
         \item [(i)] the conformal factor is constant and $F$ is $R$-Einstein.
          \item [(ii)]$F$ is a Riemannian metric.
         \end{itemize}
        \end{theorem}
        
Note that, 
the warped product of two $R$-Einstein metrics 
with different horizontal scalar curvatures is not $R$-Einstein. 
It is studied in \cite{biBesse} the special case where the conformal factor only
depends on the base of a warped product Riemannian manifold.
Thus we have the following.
\begin{theorem}\label{theo1a2}
   Let $F$ be a Finslerian metric 
  on a cylinder $\mathbb{R}\times \stackrel{\oldstylenums{2}}{M}$ of dimension $n\geq3$ and $\stackrel{\oldstylenums{2}}{F}$ 
  a Finslerian metric on $\stackrel{\oldstylenums{2}}{M}$. Let $u$ be a $C^{\infty}$ 
function on $\mathbb{R}\times \stackrel{\oldstylenums{2}}{M}$ such that $u(t,x)=u(t)$ for every $t\in \mathbb{R}$
  and $x\in \stackrel{\oldstylenums{2}}{M}$.
  Then $F$ is locally horizontally conformal to an Einstein metric $\widetilde{F}$, with $\widetilde{F}=e^uF$, 
  if and only if one of the following cases occurs:
  \begin{itemize}
    \item [(i)] $u$, in the conformal factor $e^u$, is a constant function. 
    \item [(ii)]$e^{u(t,x)}=\alpha e^{s^*t}+\beta e^{-s^*t}$, where $s^*=\sqrt{\frac{\textbf{Scal}_F^H}{(n-1)(n-2)}}$, 
    for some real constants $\alpha$ and $\beta$, 
    and $\stackrel{\oldstylenums{2}}{F}$ is horizontally Ricci-constant with positive horizontal scalar curvature $\textbf{Scal}_F^H$.
    \item [(iii)]$e^{u(t,x)}=\mu  cos\Big(\sqrt{\frac{-\textbf{Scal}_F^H}{(n-1)(n-2)}}t\Big)+\gamma sin\Big(\sqrt{\frac{-\textbf{Scal}_F^H}{(n-1)(n-2)}}t\Big)$, for some real constants $\mu$ and $\gamma$,  
    and $\stackrel{\oldstylenums{2}}{F}$ is horizontally Ricci-constant with negative horizontal scalar curvature $\textbf{Scal}_F^H$.
  \end{itemize}
\end{theorem}
For non-warped product Finslerian metrics, we obtain the following.
\begin{theorem}\label{theo1a3}
   A Finslerian metric $F$ on a $3$-dimensional (respectively $4$-dimensional) manifold 
   is locally conformally $R$-Einstein   
  if and only if the conformal factor is constant 
    and the Finslerian analogous of Cotton-York (respectively of Bach) tensor vanishes.
 \end{theorem}
 
 The rest of this paper is organised as follows. 
In Section \ref{Section2},  
we give some basic notions on Finslerian manifolds.
 The Section \ref{Section3} is devoted to study of Finslerian $R$-Einstein metrics.
 In the Section \ref{Section4} we derive Finslerian locally conformal $R$-Einstein equation. 
The Theorem \ref{theo1a1} is proved in Section \ref{Section5}.
An intrinsic theory on Finslerian warped product metrics is developped in Section \ref{Section6} and the Theorem
\ref{theo1a2}. Finally the Theorem \ref{theo1a3} is proved in Section \ref{Section7}. 
  
   \section{Preliminaries}\label{Section2}
   Throughout this paper, all manifolds are assumed to be connected and, 
   all manifolds and mappings are supposed to be differentiable of classe $C^{\infty}$. 
However, our results presented hold under the differentiability of class  $C^4$.
 Let $M$ be an $n-$dimensional manifold. We denote by $T_xM$ the tangent space at 
 $x\in M$ and by $TM:=\bigcup_{x\in M}T_xM$ the tangent bundle of $M$.
 Set $\mathring{T}M=TM\backslash\{0\}$ and $\pi:TM\longrightarrow M, \pi(x,y)\longmapsto x$ the natural projection. 
 Let $(x^1,...,x^n)$ be a local coordinate 
 on an open subset $U$ of $M$ and $(x^1,...,x^n,y^1,...,y^n)$ be the local coordinate 
 on $\pi^{-1}(U)\subset TM$. 
  The local coordinate system $(x^i)_{i=1,...,n}$ produces the local coordinate bases  
  $\{\frac{\partial}{\partial x^i}\}_{i=1,...,n}$ and $\{dx^i\}_{i=1,...,n}$ respectively, for 
  $TM$ and cotangent bundle $T^*M$.
  We use Einstein summation convention: repeated
upper and lower indices will automatically be summed unless otherwise will be noted.
 \begin{definition}\label{defi1} A function $F:TM\longrightarrow [0,\infty)$ is called a \textit{Finsler metric} 
 on $M$ if :
 \begin{itemize}
  \item [(1)] $F$ is $C^{\infty}$ on the entire slit tangent bundle $\mathring{T}M$,
  \item [(2)] $F$ is positively $1$-homogeneous on the fibers of $TM$, that is \\
  $\forall c>0,~F(x,cy)=cF(x,y),$
  \item [(3)] the Hessian matrix $(g_{ij}(x,y))_{1\leq i,j\leq n}$ with elements
  \begin{eqnarray}\label{1}
   g_{ij}(x,y):=\frac{1}{2}\frac{\partial^2F^2(x,y)}{\partial y^i\partial y^j}
  \end{eqnarray}
  is positive definite at every point $(x,y)$ of $\mathring{T}M$.
 \end{itemize}
 \end{definition}
\begin{rem}
$
 F(x,y)\neq 0 \text{   for all   } x\in M \text{   and for every   } y\in T_xM\backslash\{0\}.
 $
\end{rem}

Consider the differential map $\pi_*$  of the submersion 
 $\pi:\mathring{T}M \longrightarrow M$.
 The vertical subspace of $T\mathring{T}M$ is defined by 
 $\mathcal{V}:=ker({\pi_*})$ 
 and is locally spanned by the set $\{F\frac{\partial}{\partial y^i}, 1\leq i\leq n\}$,
  on each $\pi^{-1}(U)\subset \mathring{T}M$.
  
  An horizontal subspace $\mathcal{H}$ of $T\mathring{T}M$ is by definition any complementary to 
 $\mathcal{V}$. The bundles $\mathcal{H}$ and $\mathcal{V}$ give a smooth splitting 
 \begin{eqnarray}\label{decomposition}
T\mathring{T}M=\mathcal{H}\oplus\mathcal{V}.  
 \end{eqnarray}
An Ehresmann connection is a selection of a horizontal subspace $\mathcal{H}$ of $T\mathring{T}M$.
As explain in \cite{biMbatakou2016}, $\mathcal{H}$ can be canonically defined from the geodesics equation.
\newpage
\begin{definition}\label{defi2} Let $\pi:\mathring{T}M\longrightarrow M$ be the 
restricted projection.
\begin{itemize}
 \item [(1)]An Ehresmann-Finsler connection of $\pi$
is the subbundle $\mathcal{H}$ of $T\mathring{T}M$ given by 
 \begin{eqnarray}
  \mathcal{H}:=ker \theta,
 \end{eqnarray}
where $\theta:T\mathring{T}M\longrightarrow \pi^*TM$ is the bundle morphism defined 
by 
    \begin{eqnarray}
   \label{03b}
   \theta=\frac{\partial}{\partial x^i}\otimes \frac{1}{F}(dy^i+N_j^idx^j).
  \end{eqnarray}
  where $N_j^i(x,y):=\frac{\partial G^i(x,y)}{\partial y^j}$ with 
  $
  G^i(x,y):=\frac{1}{4}g^{il}\left(\frac{\partial g_{jl}}{\partial x^k}+\frac{\partial g_{lk}}{\partial x^j}
  -\frac{\partial g_{jk}}{\partial x^l}\right)y^jy^k.$ 
\item [(2)]The form $\theta:T\mathring{T}M\longrightarrow \pi^*TM$ induces a linear map 
\begin{eqnarray}
 \theta|_{(x,y)}:T_{(x,y)}\mathring{T}M\longrightarrow T_xM,
\end{eqnarray} 
  for each point $(x,y)\in \mathring{T}M$; where $x=\pi(x,y)$. \\
 The vertical lift of a section $\xi$ of $\pi^*TM$ is a unique section $\textbf{v}(\xi)$ of $T\mathring{T}M$ 
such that for every $(x,y)\in \mathring{T}M$, 
\begin{eqnarray}
 \pi_* (\textbf{v}(\xi))|_{(x,y)}=0_{(x,y)}\text{     and     }\theta (\textbf{v}(\xi))|_{(x,y)}=\xi_{(x,y)}.
\end{eqnarray}
\item [(3)]The differential projection $\pi_*:T\mathring{T}M\longrightarrow \pi^*TM$ induces a linear map 
\begin{eqnarray}
 \pi_*|_{(x,y)}:T_{(x,y)}\mathring{T}M\longrightarrow T_xM,
\end{eqnarray} 
  for each point $(x,y)\in \mathring{T}M$; where $x=\pi(x,y)$. \\
 The horizontal lift of a section $\xi$ of $\pi^*TM$ is a unique section $\textbf{h}(\xi)$ of $T\mathring{T}M$ 
such that for every $(x,y)\in \mathring{T}M$,
\begin{eqnarray}
 \pi_* (\textbf{h}(\xi))|_{(x,y)}=\xi_{(x,y)}\text{     and     }\theta (\textbf{h}(\xi))|_{(x,y)}=0_{(x,y)}.
\end{eqnarray}
\end{itemize}
 \end{definition}

 We have the following.
  \begin{definition}\label{defi3b} 
  A Finslerian tensor field $T$ of type $(q,0;p_1,p_2)$
  on $\mathring{T}M$ is a $C^{\infty}$ section of the tensor bundle 
  \begin{eqnarray}
  	\underbrace{\pi^*T^*M\otimes...\otimes\pi^*T^*M}_{p_1-times}
  	\otimes \underbrace{T^*\mathring{T}M\otimes...\otimes T^*\mathring{T}M}_{p_2-times}
  	\otimes\bigotimes^q\pi^*TM.  
  \end{eqnarray}
  $\big(p_1,p_2$ and $q\in\mathbb{N}\big)$ which is $C^{\infty}(\mathring{T}M,\mathbb{R})$
  -linear in each argument.
 \end{definition}
 \begin{rem}
 	In a local chart, $$T=T_{i_1...i_{p_1}j_1...j_{p_2}}^{k_1...k_q}\partial_{k_1}\otimes ... \otimes\partial_{k_q}
 	\otimes dx^{i_1}\otimes...\otimes dx^{i_{p_1}}\otimes \varepsilon^{j_{1}}\otimes ... \otimes\varepsilon^{j_{p_2}}$$
 	where $(\partial_{k_1}\otimes ... \otimes\partial_{k_q}
 	\otimes dx^{i_1}\otimes...\otimes dx^{i_{p_1}}\otimes \varepsilon^{j_{1}}\otimes ... \otimes\varepsilon^{j_{p_2}})
 	_{k\in\{1,...,n\}^q, i\in\{1,...,n\}^{p_1}, j\in\{1,...,n\}^{p_2}}$ is a basis section of this tensor 
 	and, the $\partial_{k_r}:=\frac{\partial}{\partial x^{k_r}}$ as well as $\varepsilon^{j_{s}}$
 	are respectively the basis sections for $\pi^*TM$ and $T^*\mathring{T}M$ dual of $T\mathring{T}M$.
 \end{rem}

  \begin{exe}
   \begin{itemize}
   \item [(1)]A vector field $X$ on $\mathring{T}M$ is of type $(0,0;0,1)$.
   \item [(2)]A section $\xi$ of $\pi^*TM$ is of type $(0,0;1,0)$.
   \item [(3)]The fundamental tensor $g$ is of type $(0,0;2,0)$.
   \end{itemize}
  \end{exe}
The following lemma defines the Chern connection on $\pi^*TM$.
\begin{lemma}\cite{biMbatakou2015}\label{lem1}
  Let $(M,F)$ be a Finslerian manifold and $g$ its fundamental tensor. 
 There exists a unique linear connection $\nabla$ 
 on the vector bundle $\pi^*TM$ such that, for all 
 $X,Y\in \chi(\mathring{T}M)$ and for every $\xi, \eta\in\Gamma(\pi^*TM)$, one has the following properties:
 \begin{itemize}
  \item [(i)]
  $\nabla_X\pi_*Y-\nabla_Y\pi_*X=\pi_*[X,Y],$
  \item [(ii)] 
  $X(g(\xi,\eta))=g(\nabla_X\xi,\eta)+g(\xi,\nabla_X\eta)+2\mathcal{A}(\theta(X),\xi,\eta)$ \\
  where $\mathcal{A}:=\frac{F}{2}\frac{\partial g_{ij}}{\partial y^k}dx^i\otimes dx^j\otimes dx^k$ is the Cartan tensor.
 \end{itemize}
 \end{lemma}

 One has
\begin{eqnarray}
 \nabla_{\frac{\delta}{\delta x^j}}\frac{\partial}{\partial x^k}=\Gamma_{jk}^i\frac{\partial}{\partial x^i}, \text{   }
\Gamma_{jk}^i:=\frac{1}{2}g^{il}\left(\frac{\delta g_{jl}}{\delta x^k}+\frac{\delta g_{lk}}{\delta x^j}
 -\frac{\delta g_{jk}}{\delta x^l}\right)\label{8775244}
\end{eqnarray}
where
 \begin{eqnarray}
  \left\{\frac{\delta}{\delta x^i}:=\frac{\partial}{\partial x^i}-N_i^j\frac{\partial}{\partial y^j}
  =\textbf{h}(\frac{\partial}{\partial x^i})\right\}_{i=1,...,n}\text{     with     }
 N_j^i=\Gamma_{jk}^iy^k.\label{562564654}
 \end{eqnarray}
 The generalized Cartan connection on $\pi^*TM$ is given as follows.
\begin{lemma}\cite{biMbatakou2015}\label{lem2}
  Let $(M,F)$ be a Finslerian manifold and $g$ its fundamental tensor. 
 There exists a unique linear connection ${}^c\nabla$ 
 on the vector bundle $\pi^*TM$ such that, for all 
 $X,Y\in \chi(\mathring{T}M)$ and for every $\xi, \eta,\nu\in\Gamma(\pi^*TM)$, one has the following properties:
 \begin{itemize}
  \item [(i)] 
   ${}^c\nabla_X\pi_*Y-{}^c\nabla_Y\pi_*X=\pi_*[X,Y]
  +\left(\mathcal{A}(\theta(X),\pi_*Y,\bullet)\right)^\sharp-\left(\mathcal{A}(\pi_*X,\theta(Y),\bullet)\right)^\sharp,$
   \item [(ii)] $X(g(\xi,\eta))=g({}^c\nabla_X\xi,\eta)+g(\xi,{}^c\nabla_X\eta)$ 
  where $\mathcal{A}$ is the Cartan tensor and $(~~~)^{\sharp}$ the section of $\pi^*TM$ dual to $\mathcal{A}$ defined by 
  $g\left(\mathcal{A}(\xi,\eta,\bullet\right)^\sharp,\nu)=\mathcal{A}(\xi,\eta,\nu)$.
 \end{itemize}
 \end{lemma}
\section{Finslerian $R$-Einstein metrics}\label{Section3}
\subsection{First curvature $R$ associated with the Chern connection or the Cartan connection}
\begin{definition}
 The full curvature of a linear connection $\nabla$
 on the vector bundle $\pi^*TM$ over the manifold $\mathring{T}M$ is the application
$$ \phi:\begin{matrix}
 \chi(\mathring{T}M)\times\chi(\mathring{T}M)\times\Gamma(\pi^*TM)&\to&\Gamma(\pi^*TM)\\
 (X,Y,\xi)&\mapsto
 &\phi(X,Y)\xi=\nabla_X\nabla_Y\xi-\nabla_Y\nabla_X\xi-\nabla_{[X,Y]}\xi.
 \end{matrix}$$
 \end{definition}
 By the relation (\ref{decomposition}), we have
  \begin{eqnarray}
   \nabla_X=\nabla_{\hat{X}}+\nabla_{\check{X}},\nonumber
  \end{eqnarray}
where $X=\hat{X}+\check{X}$ with $\hat{X}\in\Gamma(\mathcal{H})$ and $\check{X}\in\Gamma(\mathcal{V})$.

 Using the metric $F$, one can define the full curvature  
 of $\nabla$ as:
\begin{eqnarray}
 \Phi(\xi,\eta,X,Y)&=&g(\phi(X,Y)\xi,\eta)\nonumber\\
                   &=&g(\phi(\hat{X},\hat{Y})\xi+\phi(\hat{X},\check{Y})\xi
                   +\phi(\check{X},\hat{Y})\xi+\phi(\check{X},\check{Y})\xi,\eta)\nonumber\\
                   &=&\textbf{R}(\xi,\eta,X,Y)+\textbf{P}(\xi,\eta,X,Y)+\textbf{Q}(\xi,\eta,X,Y),\nonumber
\end{eqnarray}
where 
  $
      \textbf{R}(\xi,\eta,X,Y)=g(\phi(\hat{X},\hat{Y})\xi,\eta),$ 
   $\textbf{P}(\xi,\eta,X,Y)=g(\phi(\hat{X},\check{Y})\xi,\eta)
                   +g(\phi(\check{X},\hat{Y})\xi,\eta)$ and $\textbf{Q}(\xi,\eta,X,Y)=g(\phi(\check{X},\check{Y})\xi,\eta)$
                   are respectively the \textit{first (horizontal) curvature}, 
\textit{mixed curvature} and \textit{vertical} curvature.

In particular, if $\nabla$ is the Chern connection, the $\textbf{Q}$-curvature vanishes.
\begin{pro}\label{pro01b}
 Let ${}^c\Phi$ be the full curvature tensor associated with the Cartan connection and $\Phi$ be the full curvature tensor associated with the Chern connection. Then in the horizontal direction,
 ${}^c\Phi=\Phi$.
 \end{pro}
 \begin{proof}If $X,Y\in \mathcal{H}$ then  
 $X=\hat{X}=\hat{X}^k\frac{\delta}{\delta x^k}$ and $Y=\hat{Y}=\hat{Y}^r\frac{\delta}{\delta x^r}$. By the relation (\ref{03b}), we get
 \begin{eqnarray}
  \theta(\hat{X})&=&
  \left[\frac{\partial}{\partial x^i}\otimes \frac{1}{F}(dy^i+N_j^idx^j)\right](\hat{X}^k\frac{\delta}{\delta x^k})\nonumber\\
  &=&-\frac{\hat{X}^k}{F}N_k^s\delta_s^i\frac{\partial}{\partial x^i}
  +\frac{\hat{X}^k}{F}N_j^i\delta_k^j\frac{\partial}{\partial x^i}\nonumber\\
  &=&0.\label{kjs}
 \end{eqnarray}
Using the relation (\ref{kjs}), the both connections (${}^c\nabla$ and $\nabla$) verify the equations $(i)$ and $(ii)$ in the Lemma \ref{lem1} and the Lemma \ref{lem2}. That is, for horizontal vectors fields on $\mathring{T}M$, ${}^c\nabla$ and $\nabla$ are torsion-free and are compatible with respect to the Finslerian metric $F$. Thus, from the Lemma \ref{lem1} and Lemma \ref{lem2}, ${}^c\nabla=\nabla$.
 \end{proof}
\subsection{$R$-Einstein metric}
With respect to the Chern connection or the Cartan connection, we have the following.
\begin{definition}\label{defi4b}
   The horizontal Ricci tensor $\textbf{Ric}_F^H$ and the horizontal scalar curvature $\textbf{Scal}_F^H$ of $(M,F)$ 
   are respectively defined by
   \begin{eqnarray}
   \textbf{Ric}_F^H(\xi,X)&:=&\Phi\Big(\xi,\frac{\partial}{\partial x^i},X,\textbf{h}(\frac{\partial}{\partial x^j})\Big)g^{ij},\label{03as}\\
   \textbf{Scal}_F^H&:=&trace_{\underline{g}}\Big(\textbf{Ric}_F^H\Big), ~~~\underline{g}:=\pi^*g.\label{Sd1}
    \end{eqnarray}
\end{definition}
\begin{rem}\label{remfdgf}
Let $l:=\frac{y^i}{F}\frac{\partial}{\partial x^i}$ be the distinguish section for $\pi^*TM$.
The tensor $\textbf{Ric}_F^H$ can be expressed in term of the classical Akbar-Zadeh Ricci curvatures \cite{biBing}
 $\mathcal{R}ic$ and $\textbf{Ric}_{ij}$ as follows.
 \begin{eqnarray}\label{003b5b}
  \textbf{Ric}_F^H(l,\textbf{h}(l))
  &\stackrel{(\ref{03as})}{=}&
  g^{ij}\textbf{R}(l,\partial_i,
  \textbf{h}(l),\hat{\partial}_j)\nonumber\\
  &=&g^{ij}l^l\textbf{R}(\partial_l,\partial_i,
  ,\hat{\partial}_k,\hat{\partial}_j)l^k\nonumber\\
  &=&
  \mathcal{R}ic\nonumber\\
  &=&l^il^j\textbf{Ric}_{ij}.\nonumber
 \end{eqnarray}
\end{rem}

It is known \cite{biBao2}, $F$ is Einstein if there exists a $C^{\infty}$ function $k$ on $M$ such that 
\begin{eqnarray}\label{85a}
  \mathcal{R}ic=(n-1)k.
\end{eqnarray}

\begin{rem}\label{lem1001b}
 If $F$ is a Finslerian Einstein metric on an $n$-dimensional manifold $M$ then its associated horizontal scalar curvature 
 is a function on $M$. That is, for any $(x,y)$ of $\mathring{T}M$, $\textbf{Scal}_F^H(x,y)=n(n-1)k(x)$.
\end{rem}
Now, we introduce the following.
\begin{definition}\cite{biNibaruta2}\label{defi10} A Finslerian metric $F$ on an $n$-dimensional manifold is
$R$-Einstein if
  \begin{eqnarray}
   \textbf{Ric}_F^H=\frac{1}{n}\textbf{Scal}_F^H\underline{g}.\label{EinstC1}
  \end{eqnarray}
  \end{definition}
  \begin{rem}\label{rem10}
 If $F$ satisfies (\ref{EinstC1})
 for a constant function $\textbf{Scal}_F^H$ (respectively for $\textbf{Scal}_F^H\equiv0$) then $F$ is said to be  
   horizontally Ricci-constant (respectively, $F$ is called horizontally Ricci-flat metric).
 \end{rem}
\subsection{Schur's type lemma}
\begin{definition}\label{defi8}
 Let $\textbf{T}$ be a $(0,0;p_1,p_2)$-tensor on $(M,F)$ and $X\in T\mathring{T}M$.
 The covariant derivative of $\textbf{T}$ in the direction of $X$ 
 is given by the following formula:
 \begin{eqnarray}
  \left(\nabla_X\textbf{T}\right)\left(\xi_1,...,\xi_{p_1},X_1,...,X_{p_2}\right)
  &:=&X\left(\textbf{T}\left(\xi_1,...,\xi_{p_1},X_1,...,X_{p_2}\right)\right)\nonumber\\
  &&-\sum_{i=1}^{p_1}\left[\textbf{T}\left(\xi_1,...,\nabla_{X}\xi_i,...,\xi_{p_{1}},X_1,...,X_{p_{2}}\right)\right]\nonumber\\
  &&-\sum_{j=1}^{p_2}\left[\textbf{T}(\xi_1,...,\xi_{p_1},X_1,...,\textbf{h}(\nabla_X\pi_*X_j),...,X_{p_2})\right]\nonumber\\
  &&-\sum_{j=1}^{p_2}\left[\textbf{T}(\xi_1,...,\xi_{p_1},X_1,...,\textbf{h}(\nabla_X\theta(X_j)),...,X_{p_2})\right]\nonumber\\
&&-\sum_{j=1}^{p_2}\left[\textbf{T}(\xi_1,...,\xi_{p_1},X_1,...,\textbf{v}(\nabla_X\pi_*X_j),...,X_{p_2})\right]\nonumber\\
&&-\sum_{j=1}^{p_2}\left[\textbf{T}(\xi_1,...,\xi_{p_1},X_1,...,\textbf{v}(\nabla_X\theta(X_j)),...,X_{p_2})\right].\nonumber\label{0033b}
 \end{eqnarray}
 \end{definition}
 
 We obtain the Finslerian horizantal Bianchi identity given in the following.
\begin{lemma}\label{lem3}
 If $\xi,\eta\in\Gamma(\pi_*TM)$ and $X,Y,Z\in \chi(\mathring{T}M)$ then
 \begin{eqnarray}
  \left(\nabla_Z\textbf{R}\right)(\xi,\eta,X,Y)+  \left(\nabla_X\textbf{R}\right)(\xi,\eta,Y,Z)
  +  \left(\nabla_Y\textbf{R}\right)(\xi,\eta,Z,X)=0.\nonumber \label{00877}
 \end{eqnarray}
\end{lemma}
\begin{proof}
 The Lemma \ref{lem3} is obtained from the symmetry of $\nabla$ and the Jacobi identity and by the Definition \ref{defi8} applied 
 to the first curvature $\textbf{R}$.
\end{proof}

 We prove a Schur lemma for $\textbf{Scal}_F^H$. 
\begin{lemma}\label{proShu}
   If $F$ is horizontally an Einstein metric on a connected manifold of dimension $n\geq3$ then its horizontal scalar curvature 
   is constant.
  \end{lemma}
\begin{proof}
If $F$ is horizontally an Einstein metric then the relation (\ref{EinstC1}) holds.

Applying the horizontal covariant derivative on each side of the relation (\ref{EinstC1}), we obtain
 \begin{eqnarray}
  \nabla_k\textbf{Ric}_F^H(\partial_i,\hat{\partial}_j)
  &=&\frac{1}{n}\big(\nabla_k\textbf{Scal}_F^H\big)g_{ij}.\nonumber
 \end{eqnarray}
Multiplying this last equation by $g^{ik}$ we get
\begin{eqnarray}
 \nabla^i\textbf{Ric}_F^H(\partial_i,\hat{\partial}_j)
&=&\frac{1}{n}\nabla_j\textbf{Scal}_F^H.\label{jahgjfh}
\end{eqnarray}
where $\nabla^i:=g^{ik}\nabla_k$.

By contracting twice on equation (\ref{lem3}) written  in a local coordinate, we have
\begin{eqnarray} 
 \frac{1}{2}\nabla_j\textbf{Scal}_F^H&=&\nabla^i\textbf{Ric}_F^H(\partial_i,\hat{\partial}_j)\nonumber\\
 &\stackrel{(\ref{jahgjfh})}{=}&\frac{1}{n}\nabla_j\textbf{Scal}_F^H.\label{oiuyyugy}
\end{eqnarray}

When $n>2$, the equations (\ref{jahgjfh}) and (\ref{oiuyyugy}) together with the Lemma \ref{lem1001b} imply 
\begin{eqnarray} 
0&=&\nabla_j\textbf{Scal}_F^H\nonumber\\
&=&\frac{\partial \textbf{Scal}_F^H}{\partial x^j}.\nonumber
\end{eqnarray}
Hence, $\textbf{Scal}_F^H$ must be constant.
\end{proof}
 \section{Finslerian locally conformal $R$-Einstein equation}\label{Section4}
 \begin{definition}
A Finslerian metric $F$ on a manifold $M$ is
 locally conformally $R$-Einstein if each point $x\in M$ has a 
neighborhood $U$ on which there exists a $C^{\infty}$-function $u$ such that the conformal deformation $\widetilde{F}$
of $F$, with $\widetilde{F}=e^u F$, is an $R$-Einstein metric on $U$.
\end{definition}
 \begin{lemma}\label{lemgrosb}\cite{biNibaruta1}
  Let $F$ and $\widetilde{F}$ be two Finslerian metrics on an $n$-dimensional manifold $M$. 
  If $F$ is conformal to $\widetilde{F}$, with $\widetilde{F}=e^u F$,
 then the trace-free horizontal Ricci tensors $\textbf{E}_F^H$ and $\widetilde{\textbf{E}}_{\widetilde{F}}^H$ 
   are related by
\begin{eqnarray}\label{122b}
\widetilde{\textbf{E}}_{\widetilde{F}}^H 
&=& \textbf{E}_F^H
 -(n-2)\left(H_u-du\circ du\right)
-\frac{(n-2)}{n}\left(\Delta^Hu+||\triangledown u||_g^2\right)\underline{g}
+{\varPsi}_u^{\textbf{E}_F^H}
\end{eqnarray}
where ${\varPsi}_u^{\textbf{E}_F^H}$ is the $(0,0;1,1)$-tensor on $(M,F)$ given by
\begin{eqnarray}\label{123b}
{\varPsi}_u^{\textbf{E}_F^H}(\xi,X)
&:=&
(2-n)\left[\mathcal{A}(\triangledown u,\mathcal{B}(X),\xi)
 +\mathcal{A}(\triangledown u,\pi_*X,\mathcal{B}(\textbf{h}(\xi)))\right]\nonumber\\
  &&+(n-4)\mathcal{A}(\mathcal{B}(\textbf{h}(\triangledown u),\pi_*X,\xi))\nonumber\\
  &&
  +\frac{1}{n}g^{ij}\left[2(n-2)\mathcal{A}(\triangledown u,\partial_i,\mathcal{B}(\hat{\partial}_j)))\right.\nonumber\\
  &&\left.-3\mathcal{A}(\mathcal{B}(\textbf{h}(\triangledown u),\partial_j,\partial_i))\right]g(\xi,\pi_*X).\nonumber\\
    &&+g^{ij}\Big[g\left(\Theta(X,\textbf{h}(\Theta(\hat{\partial}_j,\textbf{h}(\xi)))),\partial_i\right)
  -g\left(\Theta(\hat{\partial}_j,\textbf{h}(\Theta(X,\textbf{h}(\xi))),\partial_i\right)\Big]\nonumber\\
  &&+g^{ij}\left[g\left((\nabla_X\Theta) (\hat{\partial}_j,\textbf{h}(\xi)),\partial_i\right)
  -g\left((\nabla_j\Theta) (\textbf{h}(\xi),X),\partial_i\right)\right]\nonumber\\
  &&
-\frac{1}{n}g^{ij}g^{kl}\left[
  \mathcal{A}(\mathcal{B}(\textbf{h}(\Theta_{jk}),\partial_l,\partial_i))
-\mathcal{A}(\mathcal{B}(\textbf{h}(\Theta_{kl}),\partial_j,\partial_i))
  \right]g(\xi,\pi_*X)\nonumber\\
  &&-\frac{1}{n}g^{ij}g^{kl}\left[g\left((\nabla_l\Theta)_{jk},\partial_i\right)
  -g\left((\nabla_j\Theta)_{kl},\partial_i\right)\right]g(\xi,\pi_*X),
\end{eqnarray}
 for every $\xi\in\Gamma(\pi^*TM)$ and $X\in\chi(\mathring{T}M)$ with $\Theta_{ij}=\Theta(\hat{\partial}_i,\hat{\partial}_j)$ 
 and 
$\mathcal{B}$ is the application which maps $\pi^*TM$ to $\pi^*TM$ defined by 
\begin{eqnarray}\label{djgdhdgkgd1}
 \mathcal{B}=\mathcal{B}_j^i\partial_i\otimes dx^j
\end{eqnarray}
with 
 \begin{eqnarray}
\mathcal{B}_j^i=\frac{1}{2F}\left(\nabla_ru\right)\frac{\partial(F^2g^{ir}-2y^iy^r)}{\partial y^j}.\label{degjeh8b}
 \end{eqnarray}
 \end{lemma}
 \begin{prof1a}
  Let $F$ and $\widetilde{F}$ be two conformal Finslerian metrics on a manifold 
 of dimension $n$. If $F$ is conformally $R$-Einstein then 
 $\widetilde{\textbf{E}}_{\widetilde{F}}^H$ vanishes. By the Lemma \ref{lemgrosb}, in a local chart we have
 \begin{eqnarray}\label{005b0}
  0&=& \Big[\textbf{E}_F^H
 -(n-2)\left(H_u-du\circ du\right)
-\frac{(n-2)}{n}\left(\Delta^Hu+||\triangledown u||_g^2\right)\underline{g}\Big](\partial_i,\hat{\partial}_j)\nonumber\\
&&+{\varPsi}_u^{\textbf{E}_F^H}(\partial_i,\hat{\partial}_j)
\end{eqnarray}
where
 \begin{eqnarray}\label{005b}
{\varPsi}_u^{\textbf{E}_F^H}(\partial_i,\hat{\partial}_j)
&=&(2-n)\Big[\mathcal{A}(\triangledown u,\mathcal{B}(\hat{\partial}_j),\partial_i)
 +\mathcal{A}(\triangledown u,\pi_*\hat{\partial}_j,\mathcal{B}(\textbf{h}(\partial_i)))\Big]\nonumber\\
  &&+(n-4)\mathcal{A}(\mathcal{B}(\textbf{h}(\triangledown u),\pi_*\hat{\partial}_j,\partial_i))\nonumber\\
&&
  +\frac{1}{n}g^{kl}\Big[2(n-2)\mathcal{A}(\triangledown u,\partial_k,\mathcal{B}(\hat{\partial}_l)))
  -3\mathcal{A}(\mathcal{B}(\textbf{h}(\triangledown u),
  \partial_l,\partial_k))\Big]g(\partial_i,\pi_*\hat{\partial}_j).\nonumber\\
    &&+g^{ij}\Big[g\left(\Theta(\hat{\partial}_j,\textbf{h}(\Theta(\hat{\partial}_l,\textbf{h}(\partial_i)))),\partial_k\right)
  -g\left(\Theta(\hat{\partial}_j,\textbf{h}(\Theta(\hat{\partial}_l,\textbf{h}(\partial_i))),\partial_k\right)\Big]\nonumber\\
  &&+g^{kl}\Big[g\left((\nabla_j\Theta) (\hat{\partial}_l,\textbf{h}(\partial_i)),\partial_k\right)
  -g\left((\nabla_l\Theta) (\textbf{h}(\partial_i),\hat{\partial}_j),\partial_k\right)\Big]\nonumber\\
  &&
-\frac{1}{n}g^{rs}g^{kl}\Big[
  \mathcal{A}(\mathcal{B}(\textbf{h}(\Theta_{sk}),\partial_l,\partial_r))
-\mathcal{A}(\mathcal{B}(\textbf{h}(\Theta_{kl}),\partial_s,\partial_r))
  \Big]g_{ij}\nonumber\\
  &&-\frac{1}{n}g^{rs}g^{kl}\left[g\left((\nabla_l\Theta)_{sk},\partial_r\right)
  -g\left((\nabla_s\Theta)_{kl},\partial_r\right)\right]g_{ij}.
 \end{eqnarray}
 
 Using the relation (\ref{djgdhdgkgd1}), we have
  $\mathcal{B}(\hat{\partial}_l)
                               =\mathcal{B}_{s_2}^{s_1}\delta_l^{s_2}\partial_{s_1}
                                =\mathcal{B}_l^{s_1}\partial_{s_1}$
and\\
  $\mathcal{B}(\textbf{h}(\triangledown u))=\mathcal{B}_{s_2}^{s_1}\partial_{s_1}\otimes dx^{s_2}
  (\textbf{h}(\nabla^lu\partial_l))
                              =\nabla^lu\mathcal{B}_l^{s_1}\partial_{s_1}.
 $
Thus, from (\ref{005b}), we have
 \begin{eqnarray}
  I_{1}&=&(2-n)\left[\mathcal{A}(\triangledown u,\mathcal{B}(\hat{\partial}_j),\partial_i)
 +\mathcal{A}(\triangledown u,\pi_*\hat{\partial}_j,\mathcal{B}(\textbf{h}(\partial_i)))\right]\nonumber\\
  &&+(n-4)\mathcal{A}(\mathcal{B}(\textbf{h}(\triangledown u),\pi_*\hat{\partial}_j,\partial_i))\nonumber\\
&=&(n-4)\nabla^{s_2}u\mathcal{B}_{s_2}^{s_1}\mathcal{A}_{s_1ij}
-(n-2)\Big(\nabla^{s_2}u\mathcal{B}_{i}^{s_1}\mathcal{A}_{s_1js_2}
+\nabla^{s_2}u\mathcal{B}_{j}^{s_1}\mathcal{A}_{s_1is_2}\Big),\nonumber\\
I_{2}&=&\frac{1}{n}g^{kl}\Big[2(n-2)\mathcal{A}(\triangledown u,\partial_k,\mathcal{B}(\hat{\partial}_l)))
  -3\mathcal{A}(\mathcal{B}(\textbf{h}(\triangledown u),
  \partial_l,\partial_k))\Big]g_{ij}\nonumber\\
   &=&-\frac{1}{n}g^{kl}\nabla^{s_2}u\Big[-3\mathcal{B}_{s_2}{s_1}\mathcal{A}_{s_1kl}
   +3\mathcal{B}_{k}{s_1}\mathcal{A}_{s_1ls_2}-(2n-1)\mathcal{B}_{k}{s_1}\mathcal{A}_{s_1ls_2}\Big]g_{ij},\nonumber\\
   I_{3}&=&-\frac{1}{n}g^{rs}g^{kl}\left[
  \mathcal{A}(\mathcal{B}(\textbf{h}(\Theta_{sk}),\partial_l,\partial_r))
-\mathcal{A}(\mathcal{B}(\textbf{h}(\Theta_{kl}),\partial_s,\partial_r))
  \right]g_{ij},\nonumber\
\end{eqnarray}
 \begin{eqnarray}
       I_{4}&=&g^{kl}\left[g\left(\Theta(\hat{\partial}_j,\textbf{h}(\Theta(\hat{\partial}_l,\textbf{h}(\partial_i)))),\partial_k\right)
  -g\left(\Theta(\hat{\partial}_l,\textbf{h}(\Theta(\hat{\partial}_j,\textbf{h}(\partial_i))),\partial_k\right)\right]\nonumber\\
   &=&g^{kl}\delta_i^r\delta_j^s\left[g\left(\Theta(\hat{\partial}_s,\textbf{h}
  (\Theta(\hat{\partial}_l,\textbf{h}(\partial_r)))),\partial_k\right)
  -g\left(\Theta(\hat{\partial}_l,\textbf{h}(\Theta(
  \hat{\partial}_s,\textbf{h}(\partial_r))),\partial_k\right)\right]\nonumber\\
  &=&\frac{1}{n}g^{kl}g^{rs}g_{ij}\left[g\left(\Theta(\hat{\partial}_s,\textbf{h}
  (\Theta(\hat{\partial}_l,\textbf{h}(\partial_r)))),\partial_k\right)
  -g\left(\Theta(\hat{\partial}_l,\textbf{h}(\Theta(
  \hat{\partial}_s,\textbf{h}(\partial_r))),\partial_k\right)\right]\nonumber\\
   &=&-I_{3},\nonumber\\
   I_{5}&=&-\frac{1}{n}g^{rs}g^{kl}\left[g\left((\nabla_l\Theta)_{sk},\partial_r\right)
  -g\left((\nabla_s\Theta)_{kl},\partial_r\right)\right]g_{ij},\nonumber\\
 I_{16}&=&g^{ij}\left[g\left((\nabla_X\Theta) (\hat{\partial}_j,\textbf{h}(\xi)),\partial_i\right)
  -g\left((\nabla_j\Theta) (\textbf{h}(\xi),X),\partial_i\right)\right]\nonumber\\
       &=&-I_{5}.\nonumber
 \end{eqnarray}
 Hence, putting the expressions of $I_{1},I_{2},I_{3},I_{4},I_{5}$ and $I_{6}$ 
in the right-hand side of (\ref{005b0}) we obtain the equation \ref{108cd}.
 \end{prof1a}
\begin{rem}
 The equation (\ref{108cd}) is called Finslerian locally conformal $R$-Einstein equation.
\end{rem}
 \section{Locally conformally $R$-Einstein metrics in dimensions $1$ and $2$}\label{Section5}
 \subsection{For $n=1$}
 Every Finslerian metric is conformally $R$-Einstein.
 \begin{theorem}
  Let $(M,F)$ be a Finslerian manifold of dimension one. Then $(M,F)$ is always $R$-flat.
 \end{theorem}
\begin{proof}
 This follows from the Lemma \ref{lem1} and the skewsymmetry of the curvature $\textbf{R}$.
\end{proof}
\subsection{For $n=2$: Proof of the Theorem \ref{theo1a1}}
\begin{proof}
 When $n=2$, the equation (\ref{108cd}) reduces to
   \begin{eqnarray}
 &&\textbf{E}_F(\partial_i,\hat{\partial}_j)+\frac{1}{4F}\left(\nabla_ru\nabla^{q}u\right)\frac{\partial(F^2g^{rs}-2y^ry^s)}{\partial y^q}
                  g^{kl}\mathcal{A}_{skl}g_{ij}
                 =0.\label{111}
        \end{eqnarray}
 Contracting (\ref{111}) by $g^{ij}$ yields
  \begin{eqnarray}
\frac{1}{2F}\left(\nabla_ru\nabla^{q}u\right)\frac{\partial(F^2g^{rs}-2y^ry^s)}{\partial y^q}
                  g^{kl}\mathcal{A}_{skl}
                 =0.\label{112}
        \end{eqnarray}
 Since $F$ is a Finslerian metric, $F(x,y)\neq0$ for every $(x,y)\in \mathring{T}M$
 and since $g$ is positive-definite the $g^{kl}$ functions do not vanish for any $k,l\in \{1,2\}$.
 Hence, the only solution of the equation (\ref{112}) are $\nabla_ru=0$ or $\mathcal{A}\equiv 0$.
 \begin{itemize}
  \item [(i)]If $\nabla_ru=0$, the conformal factor $e^u$ is constant. Further, if $u$ is constant then by equation (\ref{111})
  $\textbf{E}_F^H$ vanishes.
  \item [(ii)]If $\mathcal{A}\equiv 0$, by Deicke's theorem \cite{biBing}, $F$ is Riemannian. Hence, the result follows by the fact 
  that any Riemannian metric on a $2$-dimensional manifold is Einstein (see \cite{biBao2}).
 \end{itemize}
 
        Conversely, if the conformal deformation is homothetic and $F$ is horizontally locally Einstein then the relation (\ref{111}) is satisfied. 
        Thus, if $\mathcal{A}$ vanishes it is known that 
         $F$ is Riemannian and, when $n=2$, every Riemannian metric in conformally Einstein.
\end{proof}
        \begin{exe}\label{ex4c}
For a Finsler-Minkowskian metric on $\mathbb{R}^2$, $F(x,y)=F(y)$. It is known \cite{biMatveev}
the conformal deformations of $F$ 
         are of the form $\widetilde{F}=cF$ for all $c>0$. Since, the $\textbf{R}$-curvature on $\mathbb{R}^2$ vanishes, the tensor 
         $\textbf{E}_F^H$ vanishes. Then $F$ is globally (and automatically locally) conformally $R$-Einstein. 
        \end{exe}       
\section{Locally conformally $R$-Einstein metrics on a cylinder of dimension $n\geq3$}\label{Section6}
        \subsection{Warped product of Finslerian metrics}
Let $\stackrel{\oldstylenums{1}}{M}$ and $\stackrel{\oldstylenums{2}}{M}$ be two manifolds. The set of all product coordinate systems 
in $\stackrel{\oldstylenums{1}}{M}\times \stackrel{\oldstylenums{2}}{M}$ is an atlas on 
$M=\stackrel{\oldstylenums{1}}{M}\times \stackrel{\oldstylenums{2}}{M}$ called \textit{product manifold} of 
$\stackrel{\oldstylenums{1}}{M}$ and $\stackrel{\oldstylenums{2}}{M}$. 
\begin{exe}\label{exe3}The product $\mathbb{R}\times \stackrel{\oldstylenums{1}}{M}$ is called an 
 infinite cylinder over $\stackrel{\oldstylenums{1}}{M}$.
\end{exe}
\begin{exe}
 In the Example \ref{exe3}, if we replace $\mathbb{R}$ by an open interval $(1,\varepsilon)$, we obtain a finite cylinder 
 $(1,\varepsilon)\times \stackrel{\oldstylenums{1}}{M}$ over $\stackrel{\oldstylenums{1}}{M}$.
\end{exe}

\begin{rem}
 In general, the product manifold of $k$ manifolds 
 $\stackrel{\oldstylenums{1}}{M}$,..., $\stackrel{\oldstylenums{k-1}}{M}$ and $\stackrel{\oldstylenums{k}}{M}$ is the cartesian product 
 $M=\stackrel{\oldstylenums{1}}{M}\times ... \times \stackrel{\oldstylenums{k}}{M}$.
\end{rem}

Let $\stackrel{\oldstylenums{1}}{M}$ and $\stackrel{\oldstylenums{2}}{M}$ be two $C^{\infty}$ manifolds. 
For every $(x_1,x_2)\in \stackrel{\oldstylenums{1}}{M}\times \stackrel{\oldstylenums{2}}{M}$, we have the following properties 
derived from $\stackrel{\oldstylenums{1}}{M}$ and $\stackrel{\oldstylenums{2}}{M}$.
\begin{itemize}
 \item [(1)] The projections 
 \begin{eqnarray}
  \stackrel{\oldstylenums{1}}{p}&:&\stackrel{\oldstylenums{1}}{M}\times \stackrel{\oldstylenums{2}}{M}\longrightarrow \stackrel{\oldstylenums{1}}{M}
  \text{		such that		} \stackrel{\oldstylenums{1}}{p}(x_1,x_2)=x_1\nonumber\\
   \stackrel{\oldstylenums{2}}{p}&:&\stackrel{\oldstylenums{1}}{M}\times \stackrel{\oldstylenums{2}}{M}\longrightarrow \stackrel{\oldstylenums{1}}{M}
  \text{		such that		} \stackrel{\oldstylenums{2}}{p}(x_1,x_2)=x_2\nonumber
 \end{eqnarray}
are $C^{\infty}$ submersions.
  \item[(2)] $dim(\stackrel{\oldstylenums{1}}{M}\times \stackrel{\oldstylenums{2}}{M})=dim 
  \stackrel{\oldstylenums{1}}{M}+ dim \stackrel{\oldstylenums{2}}{M}$.
\end{itemize}

The warped product manifold of two Finslerian manifolds is defined as follows. 
\begin{definition}\label{defi11c}
 Let $(\stackrel{\oldstylenums{1}}{M},\stackrel{\oldstylenums{1}}{F})$ and
 $(\stackrel{\oldstylenums{2}}{M},\stackrel{\oldstylenums{2}}{F})$ be two Finslerian manifolds. 
 Let $f$ be a positive $C^{\infty}$ function on $\stackrel{\oldstylenums{1}}{M}$. 
 The warped product of $(\stackrel{\oldstylenums{1}}{M},\stackrel{\oldstylenums{1}}{F})$ and
 $(\stackrel{\oldstylenums{2}}{M},\stackrel{\oldstylenums{2}}{F})$ is a manifold 
 $M=\stackrel{\oldstylenums{1}}{M}\times_f \stackrel{\oldstylenums{2}}{M}$ equipped with the Finslerian metric
 \begin{eqnarray}\label{11c1}
  F:\mathring{T}\stackrel{\oldstylenums{1}}{M}\times \mathring{T}\stackrel{\oldstylenums{2}}{M}\longrightarrow \mathbb{R}^+
 \end{eqnarray}
 such that for any vector tangent $y\in T_xM$, 
 with $x=(x_1,x_2)\in M$ and $y=(y_1,y_2)$,
 \begin{eqnarray}\label{11c}
  F(x,y)=\sqrt{\stackrel{\oldstylenums{1}}{F^2}(x_1,\stackrel{\oldstylenums{1}}{p}_*y)+
  f^2(\stackrel{\oldstylenums{1}}{p}(x_1,x_2))\stackrel{\oldstylenums{2}}{F^2}(x_2,\stackrel{\oldstylenums{2}}{p}_*y)}
 \end{eqnarray}
  where 
 $\stackrel{\oldstylenums{1}}{p}$ 
 and $\stackrel{\oldstylenums{2}}{p}$ are respectively the projections of $\stackrel{\oldstylenums{1}}{M}\times \stackrel{\oldstylenums{2}}{M}$ 
 onto $\stackrel{\oldstylenums{1}}{M}$ and 
 $\stackrel{\oldstylenums{2}}{M}$. 
\end{definition}

\begin{rem}Let $F$ be a Finsler metric on a warped product manifold 
$\stackrel{\oldstylenums{1}}{M}\times_f \stackrel{\oldstylenums{2}}{M}$.
\begin{itemize}
 \item [(1)]$F$ is not $C^{\infty}$ on the tangent vectors 
 of the form $(y_1,0)$ nor $(0,y_2)$ at a point $(x_1,x_2)\in \stackrel{\oldstylenums{1}}{M}\times_f\stackrel{\oldstylenums{2}}{M}$.
 \item[(2)]$\stackrel{\oldstylenums{1}}{M}$ is called the base manifold while $\stackrel{\oldstylenums{2}}{M}$ is
 the fiber manifold and $f$ is called the warping function.\\ If $f\equiv 1$ then
 $(\stackrel{\oldstylenums{1}}{M}\times_f \stackrel{\oldstylenums{2}}{M},
 \sqrt{\stackrel{\oldstylenums{1}}{F^2}(x_1,\stackrel{\oldstylenums{1}}{p}_*y)+
  f^2(\stackrel{\oldstylenums{1}}{p}(x_1,x_2))\stackrel{\oldstylenums{2}}{F^2}(x_2,\stackrel{\oldstylenums{2}}{p}_*y)}$ 
  reduces to a Finslerian product manifold 
 $(\stackrel{\oldstylenums{1}}{M}\times 
 \stackrel{\oldstylenums{2}}{M},\sqrt{\stackrel{\oldstylenums{1}}{F^2}(x_1,\stackrel{\oldstylenums{1}}{p}_*y)+
  \stackrel{\oldstylenums{2}}{F^2}(x_2,\stackrel{\oldstylenums{2}}{p}_*y)}$.\\
\end{itemize}
\end{rem}

The function $F$ defined in (\ref{11c1}) and (\ref{11c}) is a Finslerian manifold. More precisely,
\begin{itemize}
 \item[(i)]$F$ is $C^{\infty}$ on 
 $\mathring{T}\stackrel{\oldstylenums{1}}{M}\times \mathring{T}\stackrel{\oldstylenums{2}}{M}$ since $\stackrel{\oldstylenums{1}}{F}$
 and $\stackrel{\oldstylenums{2}}{F}$ are respectively 
 $C^{\infty}$ on $\mathring{T}\stackrel{\oldstylenums{1}}{M}$ and $\mathring{T}\stackrel{\oldstylenums{2}}{M}$.
 \item [(ii)]$F$ is homogeneous of degree $1$ in $y=(y_1,y_2)\in T_xM$. Namely, for any $c>0$,
 \begin{eqnarray}\label{}
  F(x,cy)&\stackrel{(\ref{11c})}{=}&\sqrt{\stackrel{\oldstylenums{1}}{F^2}(x_1,(cy_1))+
  f^2(x_1)\stackrel{\oldstylenums{2}}{F^2}(x_2,(cy_2))}\nonumber\\
    &=&c\sqrt{\stackrel{\oldstylenums{1}}{F^2}(x_1,y_1)+f^2(x_1)\stackrel{\oldstylenums{2}}{F^2}(x_2,y_2)}\nonumber\\
    &=&cF(x,y).\nonumber
 \end{eqnarray}
 \item [(iii)]If $n_1$ and $n_2$ are respectively the dimensions of $(\stackrel{\oldstylenums{1}}{M},\stackrel{\oldstylenums{1}}{F})$
 and $(\stackrel{\oldstylenums{2}}{M},\stackrel{\oldstylenums{2}}{F})$,
  each element of the Hessian matrix $(g_{ij}(x,y))_{1\leq i,j\leq n_1+n_2}$ of $\frac{1}{2}F^2$,
  has the form:
 \begin{eqnarray}\label{}
  g_{ij}(x,y)&:=&\frac{\partial^2\left[\frac{1}{2}F^2(x,y)\right]}{\partial y^i\partial y^j}\nonumber\\
   &=&\frac{1}{2}\frac{\partial^2\left[\stackrel{\oldstylenums{1}}{F^2}(x_1,y_1)+
   f^2(x_1)\stackrel{\oldstylenums{2}}{F^2}(x_2,y_2)\right]}{\partial y^i\partial y^j}\nonumber\\
   &=&\frac{1}{2}\frac{\partial^2\stackrel{\oldstylenums{1}}{F^2}(x_1,y_1)}{\partial y_1^i\partial y_1^j}
   +\frac{1}{2}f^2(x_1)\frac{\partial^2\stackrel{\oldstylenums{2}}{F^2}(x_2,y_2)}{\partial y_2^i\partial y_2^j}.\nonumber
 \end{eqnarray}
 for every point $(x,y)=(x_1,x_2,y_1,y_2)\in \mathring{T}\stackrel{\oldstylenums{1}}{M}\times \mathring{T}\stackrel{\oldstylenums{2}}{M}$.
 Thus,
  \begin{eqnarray}\label{31cc}
\big(g_{ij}(x,y)\big)=
\left(\begin{array}{cc}
\big(\stackrel{\oldstylenums{1}}{g}_{ij}(x_1,y_1)\big) & 0 \\
 0  & \big(\stackrel{\oldstylenums{2}}{g}_{ij}(x_2,y_2)\big)
 \end{array} \right)
 \end{eqnarray}
 where $\stackrel{\oldstylenums{1}}{g}_{ij}(x_1,y_1):
 =\frac{1}{2}\frac{\partial^2\stackrel{\oldstylenums{1}}{F^2}(x_1,y_1)}{\partial y_1^i\partial y_1^j}$ and 
 $\stackrel{\oldstylenums{2}}{g}_{ij}(x_2,y_2):=
 \frac{1}{2}f^2(x_1)\frac{\partial^2\stackrel{\oldstylenums{2}}{F^2}(x_2,y_2)}{\partial y_2^i\partial y_2^j}.$ 
 So the fundamental tensor $g$ of $F$ is positive definite at every point
 $(x_1,x_2,y_1,y_2)\in \mathring{T}\stackrel{\oldstylenums{1}}{M}\times \mathring{T}\stackrel{\oldstylenums{2}}{M}$ 
 since $\stackrel{\oldstylenums{1}}{g}$ and $\stackrel{\oldstylenums{2}}{g}$ are.
\end{itemize}
\subsection{Curvatures associated with warped product Finslerian metrics}
 Given the submersions $\stackrel{\oldstylenums{1}}{\pi}:\mathring{T}\stackrel{\oldstylenums{1}}{M}\longrightarrow 
 \stackrel{\oldstylenums{1}}{M}$ and  $\stackrel{\oldstylenums{2}}{\pi}:\mathring{T}\stackrel{\oldstylenums{2}}{M}\longrightarrow 
 \stackrel{\oldstylenums{2}}{M}$, the fundamental tensors $\stackrel{\oldstylenums{1}}{g}$ and $\stackrel{\oldstylenums{2}}{g}$ associated 
 with $\stackrel{\oldstylenums{1}}{F}$ and $\stackrel{\oldstylenums{2}}{F}$ are Riemannian metrics on the 
 respective pulled-back tangent bundles $\stackrel{\oldstylenums{1}}{\pi^*}T\stackrel{\oldstylenums{1}}{M}$ 
 and $\stackrel{\oldstylenums{2}}{\pi^*}T\stackrel{\oldstylenums{2}}{M}$. 
 Thus, $\stackrel{\oldstylenums{1}}{\pi}$ gives rise to the Ehresmann-Finsler 
 connection 
 \begin{eqnarray}
\stackrel{\oldstylenums{1}}{\mathcal{H}}=ker \theta_1   \text{  ~~~~~		where	   ~~~~~	}
\theta_1:T\mathring{T}\stackrel{\oldstylenums{1}}{M}\longrightarrow \stackrel{\oldstylenums{1}}{\pi^*}T\stackrel{\oldstylenums{1}}{M}
 \end{eqnarray}
while $\stackrel{\oldstylenums{2}}{\pi}$ give rise to the Ehresmann-Finsler
 \begin{eqnarray}
\stackrel{\oldstylenums{2}}{\mathcal{H}}=ker \theta_2 \text{	 ~~~~~  	where	   ~~~~~	}
\theta_2:T\mathring{T}\stackrel{\oldstylenums{2}}{M}\longrightarrow \stackrel{\oldstylenums{2}}{\pi^*}T\stackrel{\oldstylenums{2}}{M}.
 \end{eqnarray}
 
 The Ehresmann-Finslerian product connection $\mathcal{H}$ is given by the product form $\theta$ of $\theta_1$ and $\theta_2$, 
 that is
 \begin{eqnarray}
  \theta=\theta_1\times\theta_2:T\mathring{T}\stackrel{\oldstylenums{1}}{M}\times T\mathring{T}\stackrel{\oldstylenums{2}}{M}\equiv
  T(\mathring{T}\stackrel{\oldstylenums{1}}{M}\times\mathring{T}\stackrel{\oldstylenums{2}}{M})\longrightarrow 
  \stackrel{\oldstylenums{1}}{\pi^*}T\stackrel{\oldstylenums{1}}{M}\times \stackrel{\oldstylenums{2}}{\pi^*}T\stackrel{\oldstylenums{2}}{M}
 \end{eqnarray}
such that 
\begin{eqnarray}
ker \theta=ker (\theta_1\times\theta_2)=ker \theta_1\oplus ker \theta_2.
\end{eqnarray}

Now, let $\stackrel{\oldstylenums{1}}{\mathcal{V}}$ and $\stackrel{\oldstylenums{2}}{\mathcal{V}}$ be the vertical subbundle of 
$T\mathring{T}\stackrel{\oldstylenums{1}}{M}$ and $ T\mathring{T}\stackrel{\oldstylenums{2}}{M}$, respectively. We obtain the 
following decomposition
\begin{eqnarray}
 T\mathring{T}(\stackrel{\oldstylenums{1}}{M}\times \stackrel{\oldstylenums{1}}{M})=
 \stackrel{\oldstylenums{1}}{\mathcal{H}}\oplus \stackrel{\oldstylenums{1}}{\mathcal{V}}
 \oplus \stackrel{\oldstylenums{2}}{\mathcal{H}}\oplus \stackrel{\oldstylenums{2}}{\mathcal{V}}.
\end{eqnarray}
\begin{pro}\label{pro2.1}
 Let $(\stackrel{\oldstylenums{1}}{M},\stackrel{\oldstylenums{1}}{F})$ and
 $(\stackrel{\oldstylenums{2}}{M},\stackrel{\oldstylenums{2}}{F})$ be two Finslerian manifolds. 
 On a warped product manifold $M=\stackrel{\oldstylenums{1}}{M}\times_f\stackrel{\oldstylenums{2}}{M}$, if 
 $\stackrel{\oldstylenums{1}}{\xi}\in\Gamma(\stackrel{\oldstylenums{1}}{\pi^*}T\stackrel{\oldstylenums{1}}{M})$, 
 $\stackrel{\oldstylenums{2}}{\xi}\in\Gamma(\stackrel{\oldstylenums{2}}{\pi^*}T\stackrel{\oldstylenums{2}}{M})$ and 
  $\stackrel{\oldstylenums{1}}{X}\in\chi(\mathring{T}\stackrel{\oldstylenums{1}}{M})$ then 
  \begin{itemize}
   \item [(i)]$\nabla_{\stackrel{\oldstylenums{1}}{X}}\stackrel{\oldstylenums{1}}{\xi}
   =\stackrel{\oldstylenums{1}}{\nabla}_{\stackrel{\oldstylenums{1}}{X}}\stackrel{\oldstylenums{1}}{\xi}$ where 
   $\stackrel{\oldstylenums{1}}{\nabla}$ is the Chern connection associated with $(\stackrel{\oldstylenums{1}}{M},\stackrel{\oldstylenums{1}}{F})$.
   \item [(ii)]$\nabla_{\stackrel{\oldstylenums{1}}{X}}\stackrel{\oldstylenums{2}}{\xi}
   =\frac{1}{f}\stackrel{\oldstylenums{1}}{X}(f)\stackrel{\oldstylenums{2}}{\xi}$.
  \end{itemize}
\end{pro}
\begin{proof}
  $(i)$ From the relation of $g$-almost compatibility of $\nabla$, we obtain 
  \begin{eqnarray}
   2g(\nabla_{\stackrel{\oldstylenums{1}}{X}}\stackrel{\oldstylenums{1}}{\xi},\stackrel{\oldstylenums{2}}{\xi})
   &=&\stackrel{\oldstylenums{1}}{X}
   [g(\stackrel{\oldstylenums{1}}{\xi},\stackrel{\oldstylenums{2}}{\xi})]
   +\stackrel{\oldstylenums{1}}{\textbf{h}(\xi)}
   [g(\stackrel{\oldstylenums{1}}{\pi}_*\stackrel{\oldstylenums{1}}{X},\stackrel{\oldstylenums{2}}{\xi})]
   -\stackrel{\oldstylenums{2}}{\textbf{h}(\xi)}
   [g(\stackrel{\oldstylenums{1}}{\xi},\stackrel{\oldstylenums{1}}{\pi}_*\stackrel{\oldstylenums{1}}{X})]\nonumber\\
   &&-g(\stackrel{\oldstylenums{1}}{\pi}_*\stackrel{\oldstylenums{1}}{X}, 
   [\stackrel{\oldstylenums{1}}{\xi},\stackrel{\oldstylenums{2}}{\xi}])
   -g(\stackrel{\oldstylenums{1}}{\xi}, 
   [\stackrel{\oldstylenums{1}}{\pi}_*\stackrel{\oldstylenums{1}}{X},\stackrel{\oldstylenums{2}}{\xi}])
   +g(\stackrel{\oldstylenums{2}}{\xi}, 
   [\stackrel{\oldstylenums{1}}{\pi}_*\stackrel{\oldstylenums{1}}{X},\stackrel{\oldstylenums{1}}{\xi}])\nonumber\\
   &&+\mathcal{A}(\theta(\stackrel{\oldstylenums{1}}{X}),\stackrel{\oldstylenums{1}}{\xi},\stackrel{\oldstylenums{2}}{\xi})
   +\mathcal{A}(\theta(\stackrel{\oldstylenums{1}}{\textbf{h}(\xi)}),\stackrel{\oldstylenums{1}}{\pi}_*\stackrel{\oldstylenums{1}}{X}
   ,\stackrel{\oldstylenums{2}}{\xi})
   -\mathcal{A}(\theta(\stackrel{\oldstylenums{2}}{\textbf{h}(\xi)}),\stackrel{\oldstylenums{1}}{\pi}_*\stackrel{\oldstylenums{1}}{X}
   ,\stackrel{\oldstylenums{1}}{\xi})\nonumber\\
   &=&0.\nonumber
  \end{eqnarray}
$(ii)$ For $\stackrel{\oldstylenums{2}}{\xi},\stackrel{\oldstylenums{2}}{\eta}
\in\Gamma(\stackrel{\oldstylenums{2}}{\pi^*}T\stackrel{\oldstylenums{2}}{M})$,
  \begin{eqnarray}
   2g(\nabla_{\stackrel{\oldstylenums{1}}{X}}\stackrel{\oldstylenums{2}}{\xi},\stackrel{\oldstylenums{2}}{\eta})
  &=&\stackrel{\oldstylenums{1}}{X} 
   [g(\stackrel{\oldstylenums{2}}{\xi},\stackrel{\oldstylenums{2}}{\eta})]\nonumber\\
   &\stackrel{(\ref{31cc})}{=}&\stackrel{\oldstylenums{1}}{X} 
   [(f\circ \stackrel{\oldstylenums{1}}{p})^2\stackrel{\oldstylenums{2}}{g}
   (\stackrel{\oldstylenums{2}}{\xi},\stackrel{\oldstylenums{2}}{\eta})]\nonumber
  \end{eqnarray}
  and the relation in $(ii)$ follows.
\end{proof}
As a direct consequence, we have
\begin{cor}\label{propdgeh}
 Let $(\stackrel{\oldstylenums{1}}{M},\stackrel{\oldstylenums{1}}{F})$ and
 $(\stackrel{\oldstylenums{2}}{M},\stackrel{\oldstylenums{2}}{F})$ be two Finslerian manifolds. 
 On a warped product manifold $M=\stackrel{\oldstylenums{1}}{M}\times_f\stackrel{\oldstylenums{2}}{M}$, if 
 $\stackrel{\oldstylenums{1}}{\xi},\stackrel{\oldstylenums{1}}{\eta}\in\Gamma(\stackrel{\oldstylenums{1}}{\pi^*}T\stackrel{\oldstylenums{1}}{M})$, 
  $\stackrel{\oldstylenums{1}}{X}, \stackrel{\oldstylenums{1}}{Y}\in\chi(\mathring{T}\stackrel{\oldstylenums{1}}{M})$ and
  $\stackrel{\oldstylenums{2}}{X}\in\chi(\mathring{T}\stackrel{\oldstylenums{2}}{M})$ then 
  \begin{itemize}
   \item [(i)]$\textbf{R}(\stackrel{\oldstylenums{1}}{\xi},\stackrel{\oldstylenums{1}}{\eta},
   \stackrel{\oldstylenums{1}}{X}, \stackrel{\oldstylenums{1}}{Y})
   =\stackrel{\oldstylenums{1}}{\textbf{R}}(\stackrel{\oldstylenums{1}}{\xi},\stackrel{\oldstylenums{1}}{\eta},
   \stackrel{\oldstylenums{1}}{X}, \stackrel{\oldstylenums{1}}{Y})$.
   \item [(ii)]$\textbf{R}(\stackrel{\oldstylenums{1}}{\xi},\stackrel{\oldstylenums{1}}{\eta},
   \stackrel{\oldstylenums{2}}{X}, \stackrel{\oldstylenums{1}}{Y})=0$.
   \end{itemize}
\end{cor}
\begin{proof}
 The proof follows from the Proposition \ref{pro2.1}.
\end{proof}
\subsection{Proof of the Theorem \ref{theo1a2}}
We consider the special case where the conformal factor only depends on the base manifold $\stackrel{\oldstylenums{1}}{M}$ of the product 
$\stackrel{\oldstylenums{1}}{M}\times \stackrel{\oldstylenums{2}}{M}$.

\begin{proof}
A Finslerian metric $F$ on a cylinder $\mathbb{R}\times \stackrel{\oldstylenums{2}}{M}$ can be written as 
$F=\sqrt{t^2+\stackrel{\oldstylenums{2}}{F}^2}$ where $\stackrel{\oldstylenums{2}}{F}$ is 
a Finslerian metric on $\stackrel{\oldstylenums{2}}{M}$. Further, if $F$ is locally conformal to the $R$-Einstein 
metric $e^{u(t)}F$, then by Proposition \ref{pro1}, we have 
\begin{itemize}
 \item [Case 1:] if $i=j=1$, that is $t=y^i=y^j$, the equation (\ref{108cd}) becomes
 \begin{eqnarray}
0&=&\textbf{E}_F(\partial_t,\hat{\partial}_t)-(n-2)\left(\nabla_t\nabla_tu-\nabla_tu\nabla_tu\right)\nonumber\\
&&+\frac{(n-2)}{n}\left(\nabla^d\nabla_du-\nabla^du\nabla_du\right)g_{tt}\nonumber\\
&&+\frac{(n-1)}{2nF}\left(\nabla_ru\nabla^{q}u\right)\frac{\partial(F^2g^{rs}-2y^ry^s)}{\partial y^q}
                  g^{kl}\mathcal{A}_{skl}g_{tt}\nonumber\\
   &=&\textbf{Ric}_F(\partial_t,\hat{\partial}_t)-\frac{1}{n}\textbf{Scal}_F^Hg_{tt}\nonumber\\
   &&-(n-2)\left(\nabla_t\nabla_tu-\nabla_tu\nabla_tu\right)\nonumber\\
&&+\frac{(n-2)}{n}g^{td}\left(\nabla_t\nabla_du-\nabla_tu\nabla_du\right)g_{tt}\nonumber\\
&&+\frac{(n-1)}{2nF}\left(\nabla_tu\nabla^{t}u\right)\frac{\partial(F^2g^{tt}-2t^2)}{\partial t}
                  g^{kl}\mathcal{A}_{tkl}\nonumber
\end{eqnarray}
since $u=u(t)$ and $y^q=t$ is a coordinate on $\mathbb{R}$. It follows that
  \begin{eqnarray}
  0 &=&-\frac{1}{n}\textbf{Scal}_F^H\nonumber\\
   &&-(n-2)\left(u^{''}-u^{'2}\right)
+\frac{(n-2)}{n}\left(u^{''}-u^{'2}\right)\nonumber\\
&&+\frac{(n-1)}{2nF}\left(u^{'2}\right)\frac{\partial(F^2-2t^2)}{\partial t}
                  g^{kl}\times 0\nonumber\\
&=&\textbf{Scal}_{\stackrel{\oldstylenums{2}}{F}}^H+(n-1)(n-2)(u^{''}-u^{'2}).
                  \label{hwgfgf3fh9}
\end{eqnarray}
 \item [Case 2:] if $i=1$ and $j\in\{2,3,...,n\}$ or $j=1$ and $i\in\{2,3,...,n\}$ that is $t\neq y^i$ 
 or $t\neq y^j$, by the Proposition \ref{propdgeh} 
 and by the fact that $u=u(t)$, each term in the left-hand side of the equation (\ref{108cd}) vanishes.
 \item [Case 3:] if $i,j\in\{2,3,...,n\}$ that is $t\neq y^i$ and $t\neq y^j$, the equation (\ref{108cd}) becomes
 \begin{eqnarray}
 0&=&\textbf{E}_F(\partial_{\alpha},\hat{\partial}_{\beta})-(n-2)\left(\nabla_{\beta}\nabla_{\alpha}u
 -\nabla_{\alpha}u\nabla_{\beta}u\right)\nonumber\\
 &&+\frac{(n-2)}{n}\left(\nabla^d\nabla_du-\nabla^du\nabla_du\right)g_{{\alpha}{\beta}}\nonumber\\
                  &&+\frac{(n-1)}{2nF}\left(\nabla_ru\nabla^{q}u\right)\frac{\partial(F^2g^{rs}-2y^ry^s)}{\partial y^q}
                  g^{kl}\mathcal{A}_{skl}g_{{\alpha}{\beta}}\nonumber\\
  &=&
\textbf{Ric}_{\stackrel{\oldstylenums{2}}{F}}(\partial_{\alpha},\hat{\partial}_{\beta})
-\frac{1}{n}\textbf{Scal}_{\stackrel{\oldstylenums{2}}{F}}^H
\stackrel{\oldstylenums{2}}{g}_{\alpha\beta}.\label{oiuhiouh}
        \end{eqnarray}
\end{itemize}

Therefore $\widetilde{F}=e^uF=e^u\sqrt{t^2+\stackrel{\oldstylenums{2}}{F}^2}$ is locally an Einstein metric if and only if
  \begin{eqnarray}\label{575543454}
0&=&\left\{ \begin{array}{ll}\textbf{Scal}_{\stackrel{\oldstylenums{2}}{F}}^H+(n-1)(n-2)(u^{''}-u^{'2})\\
\textbf{Ric}_{\stackrel{\oldstylenums{2}}{F}}(\partial_{\alpha},\hat{\partial}_{\beta})
-\frac{1}{n}\textbf{Scal}_{\stackrel{\oldstylenums{2}}{F}}^H
\stackrel{\oldstylenums{2}}{g}_{\alpha\beta}
\end{array} \right.\nonumber\\
&=& \left\{ \begin{array}{ll}
\textbf{Scal}_{\stackrel{\oldstylenums{2}}{F}}^H+(n-1)(n-2)(u^{''}-u^{'2})\\
\textbf{E}_{\stackrel{\oldstylenums{2}}{F}}(\partial_{\alpha},\hat{\partial}_{\beta}) \text{~~~for~~~} \alpha,\beta\in\{2,...,n\}.
\end{array} \right.
  \end{eqnarray}
  
 From the system (\ref{575543454}), $\stackrel{\oldstylenums{2}}{F}$ is $R$-Einstein. By the Lemma \ref{proShu}, 
  $\textbf{Scal}_{\stackrel{\oldstylenums{2}}{F}}^H=constant$. Denote this constant by $s$. The system (\ref{575543454}) 
  becomes 
 $ u^{''}-u^{'2}+\frac{s}{(n-1)(n-2)}=0 $
  or equivalently 
   \begin{eqnarray}
  u^{''}-u^{'2}+s^*=0\label{equa!c1}
  \end{eqnarray}
 where $s^*:=\frac{s}{(n-1)(n-2)}$. 
 We set $e^u=\varphi^{-1}$. Then 
 \begin{eqnarray}
  u=-ln\varphi,~~~ u^{'}=-\frac{\varphi^{'}}{\varphi}, ~~~u^{'2}=\frac{\varphi^{'2}}{\varphi^2}~~ \text{  and  }~~
  u^{''}=\frac{-\varphi^{''}\varphi+\varphi^{'2}}{\varphi^2}.\nonumber
 \end{eqnarray}
 The equation (\ref{equa!c1}) becomes
 \begin{eqnarray}
  \varphi^{''}-\varphi s^*=0.\label{eqngsgtys}
 \end{eqnarray}

 We distinguish three cases:
 \begin{itemize}
  \item [(i)] $s^*=0$. The general solution $\varphi$ of the equation (\ref{eqngsgtys}) is 
  $$\varphi(t)=c_1t+c_2.$$
  Since $\varphi(t)>0$ for every $t\in\mathbb{R}$ we necessarily have $c_1=0$ and $c_2>0$. Thus, the conformal factor satisfies $e^{u(t)}=\varphi^{-1}(t)=\frac{1}{c_2}=\alpha$ with $\alpha>0$.
  Hence, $u$ must be a constant function on $\mathbb{R}$.
  \item [(ii)] $s^*>0$. The general solution $\varphi$ of the equation (\ref{eqngsgtys}) is 
  \begin{eqnarray}
  \varphi(t)=c_3e^{\sqrt{s^*}t}+c_4e^{-\sqrt{s^*}t}.\nonumber
  \end{eqnarray}
  \item [(iii)]$s^*<0$. The general solution $\varphi$ of the equation (\ref{eqngsgtys}) is 
   $\varphi(t)=c_5 cos\Big(\sqrt{-s^*}t\Big)+c_6 sin\Big(\sqrt{-s^*}t\Big).$ 
  Since $\varphi$ is positive, one chooses $c_5$ and $c_6$ such that $c_5 cos\Big(\sqrt{-s^*}t\Big)+c_6 sin\Big(\sqrt{-s^*}t\Big)>0$ for every $t\in\mathbb{R}$.
 \end{itemize}

 Conversely, if one of the cases $(i)$, $(ii)$ and $(iii)$ is holds then $e^{u}F$ is $R$-Einstein.
\end{proof}
\begin{exe} Let $\stackrel{\oldstylenums{2}}{F}$ be a Finslerian metric on the sphere 
$\mathbb{S}^{n-1}$ with positive constant flag curvature $k=1$. We can show $\stackrel{\oldstylenums{2}}{F}$ is of horizantal scalar curvature 
$\textbf{Scal}_{\stackrel{\oldstylenums{2}}{F}}^H=(n-1)(n-2)$. Then the Finslerian metric 
$F=\sqrt{t^2+\stackrel{\oldstylenums{2}}{F}^2}$ is 
locally conformal to the $R$-Einstein metric $\widetilde{F}=cosh^{-1}tF$ for $t\in (1,\infty)$.
\end{exe}
\section{Non-product metrics locally conformally
 $R$-Einstein}\label{Section7}
 We define the following.
 \begin{definition}\label{defi12}
 Let $(M,F)$ be a Finslerian manifold of dimension $n\geq3$. The Finslerian analogous of  
  \begin{itemize} 
   \item [(1)] the Schouten tensor
  over $(M,F)$ is the $(0,0;1,1)$-tensor given by
  \begin{eqnarray}\label{120c}
    \textbf{S}_F^H:=\frac{1}{n-2}\left(\textbf{Ric}_F^H-\frac{1}{2(n-1)}\textbf{Scal}_F^H\underline{g}\right).
   \end{eqnarray}
  \item[(2)]the Weyl tensor over $(M,F)$ is the $(0,0;2,2)$-tensor defined by
   \begin{eqnarray}
    \textbf{W}_F^H:=\textbf{R}-g\odot \textbf{S}_F^H.
    \end{eqnarray}
Its components in a local coordinate are defined as follows, 
   \begin{eqnarray}
 \textbf{W}_F(\partial_l,\partial_i,\hat{\partial}_j,\hat{\partial}_k)&=&
   \textbf{R}(\partial_l,\partial_i,\hat{\partial}_j,\hat{\partial}_k)\nonumber\\
 &&-g_{lj}\textbf{S}_F(\partial_i,\hat{\partial}_k)-g_{ik}\textbf{S}_F(\partial_l,\hat{\partial}_j)\nonumber\\
 &&+g_{lk}\textbf{S}_F(\partial_i,\hat{\partial}_j)+g_{ij}\textbf{S}_F(\partial_l,\hat{\partial}_k).\label{weyl}
   \end{eqnarray}
   \item[(3)]the Cotton-York tensor of $(M,F)$ is the $(0,0;1,2)$-tensor $ \textbf{C}_F^H$ defined by
   \begin{eqnarray}
    \textbf{C}_F^H(\xi,X,Y):=\left(\nabla_X\textbf{S}_F^H\right)(\xi,Y)-\left(\nabla_Y\textbf{S}_F^H\right)(\xi,X)\label{CotSch2}
   \end{eqnarray}
   for every $\xi\in \Gamma(\pi^*TM)$ and $X,Y\in \chi(\mathring{T}M)$. In a local chart,
   \begin{eqnarray}
    \textbf{C}_F(\partial_i,\hat{\partial}_j,\hat{\partial}_k)&=&\left(\nabla_j\textbf{S}_F\right)(\partial_i,\hat{\partial}_k)
    -\left(\nabla_k\textbf{S}_F\right)(\partial_i,\hat{\partial}_j).\label{122''}
   \end{eqnarray}
   \end{itemize}
  \end{definition}
  In dimension greater than $3$, we introduce the following tensor.
    \begin{definition}\label{defi12c2}
The Finslerian analogous of Bach tensor for a Finslerian manifold $(M,F)$ is the $(1,1;0;0)$-tensor $ \textbf{B}_F^H$ defined by
   \begin{eqnarray}
    \textbf{B}_F(\partial_i,\hat{\partial}_j)=\nabla^{k}\textbf{C}_F(\partial_i,\hat{\partial}_j,\hat{\partial}_k)
    +\textbf{S}_F^{lk}\textbf{W}_F(\partial_l,\partial_i,\hat{\partial}_k,\hat{\partial}_j).
   \end{eqnarray}
    \end{definition}
  
  We have the following properties.
  \begin{lemma}\label{lem10}
    Let $(M,F)$ be a Finslerian manifold of dimension $n\geq3$. Then,
  \begin{itemize}
  \item[(1)]the Finslerian analogous of Weyl and of Cotton-York tensors are related as follows:
   \begin{eqnarray}
    \nabla^l\textbf{W}_F(\partial_l,\partial_i,\hat{\partial}_j,\hat{\partial}_k)=
    (n-3)\textbf{C}_F(\partial_i,\hat{\partial}_j,\hat{\partial}_k)\label{C-W}\nonumber
   \end{eqnarray}
   \item[(2)] if $F$ is horizontally an Einstein metric  
   then its Finslerian Cotton-York tensor vanishes
   \begin{eqnarray}
    \textbf{C}_F(\partial_i,\hat{\partial}_j,\hat{\partial}_k)=0.\nonumber\label{C02}.
    \end{eqnarray}
  \end{itemize}
  \end{lemma}
  \begin{proof}
   \begin{itemize}
    \item [(1)] Contracting the Finslerin second Bianchi identity given in Lemma \ref{lem3} we get
     $g^{sl}\Big[\nabla_j\textbf{R}_{liks}+ \nabla_k\textbf{R}_{lisj}
   + \nabla_s\textbf{R}_{lijk}\Big]=0$. 
   Equivalent
$$-\nabla_j\textbf{Ric}_F(\partial_i,\hat{\partial}_k)
  + \nabla_k\textbf{Ric}_F(\partial_i,\hat{\partial}_j)
  + \nabla^l\textbf{R}_{lijk}=0.$$
  Using this relation we have
   \begin{eqnarray}
  \nabla^l\textbf{W}_F(\partial_l,\partial_i,\hat{\partial}_j,\hat{\partial}_k)
  &=&g^{ls}\nabla_s\textbf{W}_F(\partial_l,\partial_i,\hat{\partial}_j,\hat{\partial}_k)\nonumber\\
  &\stackrel{(\ref{weyl})}{=}&g^{ls}\nabla_s\Big\{
   \textbf{R}(\partial_l,\partial_i,\hat{\partial}_j,\hat{\partial}_k)\nonumber\\
&&-\frac{1}{n-2}\left[\textbf{Ric}_F(\partial_l,\hat{\partial}_j)g_{ik}-\textbf{Ric}_F(\partial_i,\hat{\partial}_j)g_{lk}\right.\nonumber\\
   &&\left. +\textbf{Ric}_F(\partial_i,\hat{\partial}_k)g_{lj}-\textbf{Ric}_F(\partial_l,\hat{\partial}_k)g_{ij}\right]\nonumber\\
   &&+\frac{\textbf{Scal}_F^H}{(n-1)(n-2)}\left[
   g_{ij}g_{lk}-g_{ik}g_{lj}\right]\Big\}\nonumber\\
  &=&\frac{n-3}{n-2} \Big(\nabla_j\textbf{Ric}_F(\partial_i,\hat{\partial}_k)
  - \nabla_k\textbf{Ric}_F(\partial_i,\hat{\partial}_j)\Big)\nonumber\\
  &&-\frac{n-3}{2(n-1)(n-2)} \Big(\nabla_j\textbf{Scal}_F^H
   g_{ik}-\nabla_k\textbf{Scal}_F^Hg_{ij}\Big)\nonumber\\
   &\stackrel{(\ref{122''})}{=}&(n-3)\textbf{C}_F(\partial_i,\hat{\partial}_j,\hat{\partial}_k).\nonumber
   \end{eqnarray}
   \item[(2)] If $F$ is $R$-Einstein then, by Lemma \ref{proShu},  $\textbf{Scal}_F^H$ is constant. 
   Hence,
   \begin{eqnarray}
    \textbf{C}_F(\partial_i,\hat{\partial}_j,\hat{\partial}_k)
    &=&\left(\nabla_j\textbf{S}_F\right)(\partial_i,\hat{\partial}_k)
    -\left(\nabla_k\textbf{S}_F\right)(\partial_i,\hat{\partial}_j)\nonumber\\
    &=&\nabla_j\Big[\frac{1}{n-2}\left(\textbf{Ric}_F^H-\frac{1}{2(n-1)}\textbf{Scal}_F^H\underline{g}\right)
    \Big](\partial_i,\hat{\partial}_k)\nonumber\\
    &&-\nabla_k\Big[\frac{1}{n-2}\left(\textbf{Ric}_F^H-\frac{1}{2(n-1)}\textbf{Scal}_F^H\underline{g}\right)
    \Big](\partial_i,\hat{\partial}_j).\nonumber
    \end{eqnarray}
    Hence, formula (\ref{EinstC1}) implies relation (\ref{C02}).
  \end{itemize}
  \end{proof}
   \begin{lemma}\label{lem11}
    Let $F$ be a Finslerian metric on a manifold of dimension $n\geq3$. If $\widetilde{F}$ is a conformal deformation of $F$, with 
    $\widetilde{F}=e^uF$, then
  \begin{itemize} 
   \item [(1)]the horizontal Schouten tensor behaves as follows:
   \begin{eqnarray}
   \widetilde{\textbf{S}}_{\widetilde{F}}^H(\partial_i,\hat{\partial}_j)&=&\textbf{S}_F^H(\partial_i,\hat{\partial}_j)
   -\nabla_j\nabla_iu+\nabla_iu\nabla_ju
   +hg_{ij}\label{1006}
   \end{eqnarray}
   where
  \begin{eqnarray}
  h:&=&-\frac{1}{2}\nabla^ku\nabla_ku
  +\frac{\nabla^sug^{kl}}{n(n-1)}\Big[(n+8)\mathcal{B}_l^{s_1}\mathcal{A}_{sks_1}
  -2\mathcal{B}_s^{s_1}\mathcal{A}_{lks_1}\Big]\nonumber\\
  &&+\frac{1}{2n(n-1)}g^{kl}g^{rs}\Big\{\left[g(\Theta(\hat{\partial}_s,\textbf{h}(\Theta_{lr})),\partial_k)
-g(\Theta(\hat{\partial}_l,\textbf{h}(\Theta_{rs})),\partial_k)\right]\nonumber\\
&&+\Big[g((\nabla_s\Theta)_{lr},\partial_k)
-g((\nabla_l\Theta)_{rs},\partial_k)\Big]\Big\}.\nonumber
   \end{eqnarray}
  \item[(2)] if a horizontal $(1,1,0)$-tensor $\textbf{T}_F^H$ satisfies $\textbf{T}_F(\partial_i,\hat{\partial}_j)
  =\textbf{T}_F(\pi_*\hat{\partial}_j,\textbf{h}(\partial_i))$, for any $i,j=1,...,n$,
  then
   \begin{eqnarray}\label{1006a}
   \left(\widetilde{\nabla}_j\textbf{T}_F\right)(\partial_i,\hat{\partial}_k)
   &=&\Big(\nabla_j\textbf{T}_F\Big)(\partial_i,\hat{\partial}_k)\nonumber\\
   &&-2\nabla_ju\textbf{T}_{ik}-\nabla_iu\textbf{T}_{jk}-\nabla_ku\textbf{T}_{ij}\nonumber\\
   &&+g_{ij}\textbf{T}_F(\triangledown u,\partial_k)+g_{jk}\textbf{T}_F(\partial_i,\textbf{h}(\triangledown u))
   \nonumber\\
   &&-\textbf{T}_F(\partial_i,\textbf{h}(\Theta)_{jk})-\textbf{T}_F(\Theta_{ij},\hat{\partial}_k).
   \end{eqnarray}
   \item[(3)] the horizontal Cotton-York tensor behaves as follows:
   \begin{eqnarray}
    \widetilde{\textbf{C}}_{\widetilde{F}}(\partial_i,\hat{\partial}_j,\hat{\partial}_k)
    =\textbf{C}_F(\partial_i,\hat{\partial}_j,\hat{\partial}_k)
    +\textbf{W}_F(\triangledown u,\partial_i,\hat{\partial}_j,\hat{\partial}_k)
    +{\varPsi}_u^{\textbf{C}_F^H}(\partial_i,\hat{\partial}_j,\hat{\partial}_k)\nonumber
   \end{eqnarray}
   where
   \begin{eqnarray}
   {\varPsi}_u^{\textbf{C}_F^H}(\partial_i,\hat{\partial}_j,\hat{\partial}_k)
    &=&\nabla_j\Big(\nabla_iu\nabla_ku+hg_{ik}\Big)
 -\nabla_k\Big(\nabla_iu\nabla_ju+hg_{ij}\Big)\nonumber\\
   &&+\Gamma_{ik}^l\nabla_lu\nabla_ju
  -\Gamma_{ij}^l\nabla_lu\nabla_ku\nonumber\\
  &&+g_{jl}\nabla^lu\Big(
   \nabla_k\nabla_iu-hg_{ik}\Big)
   -g_{kl}\nabla^lu\Big(
   \nabla_j\nabla_iu-hg_{ij}\Big)\nonumber\\
   &&+g_{ij}\nabla_k\nabla_{\textbf{h}(\triangledown u)}u
  +g_{ik}\nabla_j\nabla_{\textbf{h}(\triangledown u)}u\nonumber\\
 &&
   -\nabla_j\nabla_{\Theta_{ik}}u+\nabla_{\Theta_{ik}}u\nabla_ju
 +\nabla_k\nabla_{\Theta_{ij}}u-\nabla_{\Theta_{ij}}u\nabla_ku.\nonumber
   \end{eqnarray}
  \end{itemize}
  \end{lemma}
  \begin{proof}
   The assertion $(1)$ in Lemma \ref{lem11} is obtained by using the relation (\ref{120c}) and the lemmas $5$ and $6$ in \cite{biNibaruta2}.

    To obtain the relation (\ref{1006a}) in Lemma \ref{lem11}, we consider a $(1,1,0)$-tensor $\textbf{T}_F^H$ on $(M,F)$. Then
   for every vector fields $X,Y$ on $\mathring{T}M$ and for any section $\xi$ of the vector bundle 
   $\pi^*TM$, we obtain
   \begin{eqnarray}
    \left(\widetilde{\nabla}_X\textbf{T}_F^H\right)(\xi,Y)=\widetilde{\nabla}_X(\textbf{T}_F^H(\xi,Y))-\textbf{T}_F^H(\widetilde{\nabla}_X\xi,Y)
    -\textbf{T}_F^H(\xi,\textbf{h}(\widetilde{\nabla}_X\pi_*Y))
   \end{eqnarray}
 where $\widetilde{\nabla}_X$ is the covariant derivative with respect to $\widetilde{F}$ in a given direction $X$. 
 We have 
 \begin{eqnarray}
    \left(\widetilde{\nabla}_X\textbf{T}_F^H\right)(\xi,Y)
    &=&\nabla_X(\textbf{T}_F^H(\xi,Y))
    -\textbf{T}_F^H\Big(\nabla_X\xi+du(\pi_*X)\xi
  +du(\xi)\pi_*X\nonumber\\
   &&-g(\pi_*X,\xi)\triangledown u+\Theta(X,\textbf{h}(\xi)),Y\Big)\nonumber\\
   && -\textbf{T}_F^H(\xi,\textbf{h}(\nabla_X\pi_*Y+du(\pi_*X)\pi_*Y\nonumber\\
   &&+du(\pi_*Y)\pi_*X-g(\pi_*X,\pi_*Y)\triangledown u+\Theta(X,Y)))\nonumber\\
   &=&\nabla_X(\textbf{T}_F(\xi,\hat{Y}))-2(\nabla_Xu)\textbf{T}_F(\xi,\hat{Y})\nonumber\\
   &&-(\nabla_{\textbf{h}(\xi)}u)\textbf{T}_F(\pi_*X,\hat{Y})-(\nabla_Yu)\textbf{T}_F(\xi,\hat{X})\nonumber\\
    &&+g(\xi,\pi_*X)\textbf{T}_F(\triangledown u,\hat{Y})+g(\pi_*X,\pi_*Y)\textbf{T}_F(\xi,\hat{\triangledown u})\nonumber\\
    &&-\textbf{T}_F(\xi,\textbf{h}(\Theta(X,Y))-\textbf{T}_F(\Theta(\textbf{h}(\xi),X),Y).\nonumber
   \end{eqnarray}
   Setting $\xi=\partial_i$, $X=\hat{\partial}_j$ and $Y=\hat{\partial}_k$, we obtain the relation.

From the these two properties, we obtain the assertion $(3)$ in the Lemma \ref{lem11}.
  \end{proof}
\subsection{Proof of the Theorem \ref{theo1a3} in dimension $n=3$}
 Let $F$ and $\widetilde{F}$ be two conformal Finslerian metric, with $\widetilde{F}=e^uF$, 
 on a manifold of dimension $n\geq3$. Then 
   \begin{eqnarray}
    \widetilde{\textbf{C}}_{\widetilde{F}}(\partial_i,\hat{\partial}_j,\hat{\partial}_k)
   &\stackrel{(\ref{122''})}{=}&\left(\widetilde{\nabla}_j{\widetilde{\textbf{S}}}_{\widetilde{F}}\right)(\partial_i,\hat{\partial}_k)
    -\left(\widetilde{\nabla}_k\widetilde{\textbf{S}}_{\widetilde{F}}\right)(\partial_i,\hat{\partial}_j)\nonumber\\
   &\stackrel{(\ref{1006a})}{=}
    &\left(\nabla_j{\widetilde{\textbf{S}}}_{\widetilde{F}}\right)(\partial_i,\hat{\partial}_k)
   -\left(\nabla_k{\widetilde{\textbf{S}}}_{\widetilde{F}}\right)(\partial_i,\hat{\partial}_j)
   \nonumber\\
  && 
   -\nabla_ju{\widetilde{\textbf{S}}}_{\widetilde{F}}(\partial_i,\hat{\partial}_k)
   +\nabla_ku{\widetilde{\textbf{S}}}_{\widetilde{F}}(\partial_i,\hat{\partial}_j)\nonumber\\
   &&+g_{ij}{\widetilde{\textbf{S}}}_{\widetilde{F}}
   (\triangledown u,\partial_k)
  -g_{ik}{\widetilde{\textbf{S}}}_{\widetilde{F}}
   (\triangledown u,\partial_j)
   \nonumber\\
   &&+{\widetilde{\textbf{S}}}_{\widetilde{F}}(\Theta_{ik},\hat{\partial}_j)
   -{\widetilde{\textbf{S}}}_{\widetilde{F}}(\Theta_{ij},\hat{\partial}_k).\nonumber\\
    &=&\nabla_j\Big[{\widetilde{\textbf{S}}}_{\widetilde{F}}(\partial_i,\hat{\partial}_k)\Big]
    -{\widetilde{\textbf{S}}}_{\widetilde{F}}\big(\nabla_j\partial_i,\hat{\partial}_k\big)\nonumber\\
    &&-{\widetilde{\textbf{S}}}_{\widetilde{F}}\big(\partial_i,\textbf{h}(\nabla_j\pi_*\hat{\partial}_k)\big)
    -\Big\{\nabla_k\Big[{\widetilde{\textbf{S}}}_{\widetilde{F}}(\partial_i,\hat{\partial}_j)\Big]\nonumber\\
    &&-{\widetilde{\textbf{S}}}_{\widetilde{F}}\big(\nabla_k\partial_i,\hat{\partial}_j\big)
    -{\widetilde{\textbf{S}}}_{\widetilde{F}}\big(\partial_i,\textbf{h}(\nabla_k\pi_*\hat{\partial}_j)\big)\Big\}\nonumber\\
    && 
   -\nabla_ju{\widetilde{\textbf{S}}}_{\widetilde{F}}(\partial_i,\hat{\partial}_k)
   +\nabla_ku{\widetilde{\textbf{S}}}_{\widetilde{F}}(\partial_i,\hat{\partial}_j)\nonumber\\
   &&+g_{ij}{\widetilde{\textbf{S}}}_{\widetilde{F}}
   (\triangledown u,\partial_k)
  -g_{ik}{\widetilde{\textbf{S}}}_{\widetilde{F}}
   (\triangledown u,\partial_j)
   \nonumber\\
   &&+{\widetilde{\textbf{S}}}_{\widetilde{F}}(\Theta_{ik},\hat{\partial}_j)
   -{\widetilde{\textbf{S}}}_{\widetilde{F}}(\Theta_{ij},\hat{\partial}_k)\nonumber
   \end{eqnarray}
   \begin{eqnarray}
   \widetilde{\textbf{C}}_{\widetilde{F}}(\partial_i,\hat{\partial}_j,\hat{\partial}_k)&=&\textbf{C}_{\widetilde{F}}(\partial_i,\hat{\partial}_j,\hat{\partial}_k)\nonumber\\
&&+\nabla_j\Big[
   -\nabla_kg_{il}\nabla^lu+\nabla_iu\nabla_ku+hg_{ik}\Big]\nonumber\\
   &&-\nabla_k\Big[
   -\nabla_jg_{il}\nabla^lu+\nabla_iu\nabla_ju+hg_{ij}\Big]\nonumber\\
   &&+\Big[
   -\nabla_j\nabla_{\nabla_k\partial_i}u+\nabla_{\nabla_k\partial_i}u\nabla_ju+hg(\nabla_k\partial_i,\pi_*\hat{\partial_j})\Big]\nonumber\\
   &&-\Big[
   -\nabla_k\nabla_{\nabla_j\partial_i}u+\nabla_{\nabla_j\partial_i}u\nabla_ku+hg(\nabla_j\partial_i,\pi_*\hat{\partial_k})\Big]\nonumber\\  
  &&-g_{jl}\nabla^lu\Big[\textbf{S}_F^H(\partial_i,\hat{\partial}_k)
   -\nabla_k\nabla_iu+\nabla_iu\nabla_ku+hg_{ik}\Big]\nonumber\\
   && 
   +g_{kl}\nabla^lu\Big[\textbf{S}_F^H(\partial_i,\hat{\partial}_j)
   -\nabla_j\nabla^lu\partial_l+\nabla_iu\nabla_ku+hg_{ij}\Big]\nonumber\\
   &&+g_{ij}\Big[\nabla^lu\textbf{S}_F^H(\partial_l,\hat{\partial}_k)
   -\nabla_k\nabla_{\textbf{h}(\triangledown u)}u\Big]\nonumber\\
  &&-g_{ik}\Big[\nabla^lu\textbf{S}_F^H(\partial_l,\hat{\partial}_j)
   -\nabla_j\nabla_{\textbf{h}(\triangledown u)}u\Big]\nonumber\\
&&+\Big[\textbf{S}_F^H(\Theta_{ik},\hat{\partial}_j)
   -\nabla_j\nabla_{\Theta_{ik}}u+\nabla_{\Theta_{ik}}u\nabla_ju\Big]\nonumber\\
  &&-\Big[\textbf{S}_F^H(\Theta_{ij},\hat{\partial}_k)
   -\nabla_k\nabla_{\Theta_{ij}}u+\nabla_{\Theta_{ij}}u\nabla_ku\Big].\nonumber
   \end{eqnarray}
   Therefore
   \begin{eqnarray}
    \widetilde{\textbf{C}}_{\widetilde{F}}(\partial_i,\hat{\partial}_j,\hat{\partial}_k)
&\stackrel{}{=}&\textbf{C}_{\widetilde{F}}(\partial_i,\hat{\partial}_j,\hat{\partial}_k)
    +\nabla^lu\textbf{W}_F(\partial_l,\partial_i,\hat{\partial}_j,\hat{\partial}_k)
    +{\varPsi}_u^{\textbf{C}_F^H}(\partial_i,\hat{\partial}_j,\hat{\partial}_k),\nonumber
   \end{eqnarray}
   where 
    \begin{eqnarray}
   {\varPsi}_u^{\textbf{C}_F^H}(\partial_i,\hat{\partial}_j,\hat{\partial}_k)
 &=&\nabla_j\Big(\nabla_iu\nabla_ku+hg_{ik}\Big)
 -\nabla_k\Big(\nabla_iu\nabla_ju+hg_{ij}\Big)\nonumber\\
   &&+\Gamma_{ik}^l\nabla_lu\nabla_ju
  -\Gamma_{ij}^l\nabla_lu\nabla_ku+g_{jl}\nabla^lu\Big(
   \nabla_k\nabla_iu-hg_{ik}\Big)\nonumber\\
   && 
   -g_{kl}\nabla^lu\Big(
   \nabla_j\nabla_iu-hg_{ij}\Big)
   +g_{ij}\nabla_k\nabla_{\textbf{h}(\triangledown u)}u
  +g_{ik}\nabla_j\nabla_{\textbf{h}(\triangledown u)}u\nonumber\\
&&
   -\nabla_j\nabla_{\Theta_{ik}}u+\nabla_{\Theta_{ik}}u\nabla_ju
 +\nabla_k\nabla_{\Theta_{ij}}u-\nabla_{\Theta_{ij}}u\nabla_ku.\nonumber
 \end{eqnarray}
 
 If $\widetilde{F}$ is $R$-Einstein then by the Lemma \ref{lem10}
 \begin{eqnarray}
  \textbf{C}_F(\partial_i,\hat{\partial}_j,\hat{\partial}_k)+\textbf{W}_F(\triangledown u,\partial_i,\hat{\partial}_j,\hat{\partial}_k)
+{\varPsi}_u^{\textbf{C}_F^H}(\partial_i,\hat{\partial}_j,\hat{\partial}_k)=0.\label{ojjj50cd}
 \end{eqnarray}

 When $n=3$, the tensor $\textbf{W}_F^H$ vanishes and hence the equation (\ref{ojjj50cd}) reduces to
  \begin{eqnarray}
  \textbf{C}_F(\partial_i,\hat{\partial}_j,\hat{\partial}_k)
  +{\varPsi}_u^{\textbf{C}_F^H}(\partial_i,\hat{\partial}_j,\hat{\partial}_k)=0.\label{ojjj50c2}
 \end{eqnarray}
 The solution of this equation is $u=constant$ and $\textbf{C}_F(\partial_i,\hat{\partial}_j,\hat{\partial}_k)$.\\
 
 Conversely, if $u=constant$ and $\textbf{C}_F^H\equiv0$ then $e^uF$ is $R$-Einstein metric.
\subsection{Proof of the Theorem \ref{theo1a3} in dimension $n=4$}
From the Lemma \ref{lem10}, if $\widetilde{F}$ is $R$-Einstein metric then $\widetilde{\textbf{C}}_{\widetilde{F}}$ vanishes. 
Then the equation (\ref{ojjj50cd}) holds.

Applying $\nabla^k$ to this equation, using the Definition \ref{defi12} and the equation (\ref{ojjj50cd}) again, we get  
 \begin{eqnarray}
 0&=&\textbf{B}_F(\partial_i,\hat{\partial}_j)
 -\textbf{S}_F^{lk}\textbf{W}_F(\partial_l,\partial_i,\hat{\partial}_k,\hat{\partial}_j)\nonumber\\
  &&- \left[\nabla^k\nabla^lu-(n-3)\nabla^ku\nabla^lu\right]\textbf{W}_F(\partial_l,\partial_i,\hat{\partial}_k,\hat{\partial}_j)\nonumber\\
  &&-\nabla^k{\varPsi}_u^{\textbf{C}_F^H}(\partial_i,\hat{\partial}_j,\hat{\partial}_k)
  +{\varPsi}_u^{\textbf{C}_F^H}(\partial_i,\hat{\partial}_j,\hat{\partial}_k).\label{10003}
 \end{eqnarray}
 Since $\widetilde{F}$ is locally an $R$-Einstein metric, the equation (\ref{108cd}) is equivalent to 
 \begin{eqnarray}
  0&=&\textbf{S}_F(\partial_i,\hat{\partial}_j)-\frac{1}{n}\textbf{J}_F^Hg_{ij}-\nabla_j\nabla_iu+\nabla_iu\nabla_ju\nonumber\\
 &&+\frac{1}{n}\left(\nabla^d\nabla_du-\nabla^du\nabla_du\right)g_{ij}\nonumber\\
                  &&+\frac{(n-1)}{2n(n-2)F}\left(\nabla_ru\nabla^{q}u\right)\frac{\partial(F^2g^{rs}-2y^ry^s)}{\partial y^q}
                  g^{kl}\mathcal{A}_{skl}g_{ij}\nonumber
 \end{eqnarray}
 where $\textbf{J}_F^H$ is the trace of $\textbf{Scal}_F^H$.
Raising both indices and applying 
$\textbf{W}_F(\partial_l,\partial_i,\hat{\partial}_j,\hat{\partial}_k)$ to this equation, using the relation (\ref{C-W}) 
in Lemma \ref{lem11} and the equation (\ref{10003}) we obtain
\begin{eqnarray}\label{ojjj50c2}
 0&=&\textbf{B}_F(\partial_i,\hat{\partial}_j)+(n-4)\textbf{W}_F(\triangledown u,\partial_i,\hat{\partial}_j,\triangledown u)\nonumber\\
 &&
+\left[(n-3)\nabla^ku- \nabla^k\right]{\varPsi}_u^{\textbf{C}_F^H}(\partial_i,\hat{\partial}_j,\hat{\partial}_k).\nonumber
\end{eqnarray}

Therefore, in dimension $n=4$, we have 
 $\textbf{B}_F(\partial_i,\hat{\partial}_j)
+\Big(\nabla^ku- \nabla^k\Big){\varPsi}_u^{\textbf{C}_F^H}(\partial_i,\hat{\partial}_j,\hat{\partial}_k)=0.$

Conversely, if $u=constant$ and $\textbf{B}_F^H\equiv0$ then $e^uF$ is $R$-Einstein metric.


\begin{thebibliography}{10}
\mark{RReferences}
 
 \bibitem{biAntonelli}
    P. L. Antonelli, R. S. Ingarden and M. Matsumoto, 
    {\em The Theory of Sprays and Finsler Spaces with Applications in Physics and Biology},
     Kluwer Academic Publishers, Dordrecht, 1993. 


\bibitem{biBao}{D. Bao, S.-S. Chern and Z. Shen}, \textit{An Introduction to Riemannian-Finsler Geometry}, 
Springer-Verlang New York,  (2000), 1-192.
     

 \bibitem{biBesse}{A. L. Besse}, 
\textit{Einstein Manifolds}, Springer-Verlag Berlin Heidelberg, (1987), 1-528.

 \bibitem{biBrinkmann1} {H. W. Brinkmann}, \textit{Riemann spaces conformal to Einstein spaces}, Proc. Nat. Acad. Sci. USA \textbf{9}
(1923), 172-174. 
 

\bibitem{biKNT} {C. N. Kozameh, E. T. Newman and K. P. Tod}, \textit{Conformal Einstein Spaces}, 
Gen. Rel. and Grav., Vol. 17 (1985), 343-344.

 \bibitem{biListing1} {M. Listing}, \textit{Conformal Einstein Spaces in $N$-dimensions}, Annals of
Global Analysis and Geometry, 20 (2001), 183-197.
 
 

 \bibitem{biMatveev}{V. S. Matveev, H. B. Rademacher, M. Troyanov and A. Zeghib},
        \textit{Finsler conformal Lichnerwicz-Obata Conjecture}, Annales de l'Institut de Fourier, Grenoble
        59, \textbf{3} (2009), 937-949.

 \bibitem{biMbatakou2015}{J.-S. Mbatakou}, 
\textit{Intrinsic proofs of the existence of generalized Finsler connections}, \textit{Int. Electron. J. Geom.}, 8, \textbf{1} (2015), 1-13.

 \bibitem{biNibaruta1}{G. Nibaruta, S. Degla and L. Todjihounde}, 
\textit{Prescribed Ricci tensor in Finslerian conformal class}, Balkan J. Geom. Appl., 23, 2 (2018), 41-55.

\bibitem{biNibaruta2}{G. Nibaruta, S. Degla and L. Todjihounde}, 
\textit{Finslerian Ricci Deformation and Conformal Metrics}, J. Appl. Math. Phys., \textbf{6} (2018), 1522-1536.

 \bibitem{biGover} {A. R. Gover and P. Nurowski},  
 {Obstructions to conformally Einstein metrics in $n$ dimensions}, Journal of Geometry and Physics, www.elsevier.com/locate/jpg. (2004).

 \bibitem{biRademacher} {W. K\"{u}hnel and H.-B. Rademacher}, \textit{Conformally Einstein product spaces},
 Differ. Geom. Appl., \textbf{49} (2016), 65-96.
 
 \bibitem{biBing} Y.-B. Shen and Z. Shen,
    \textit{Introduction to Modern Finsler Geometry}, 
 Higher Education Press Limited Company and World Scientific Publishing Co. Pte. Ltd. (2016), 1-58.
 
 \bibitem{biSzekeres} {P. Szekeres}, \textit{Spaces Conformal to a Class of Spaces in General Gelativity},
Proc. Roy. Soc. London, Ser. A., \textbf{274}
 (1963), 206-212.
\end{thebibliography}
 \end{document}